\documentclass[11pt]{amsart}
\usepackage{amsmath,amssymb,amsthm,amsfonts,hhline,color,multicol,multirow,booktabs,mathtools,float}
\usepackage{etoolbox}
\usepackage{setspace}
\usepackage[table,xcdraw,svgnames]{xcolor}
\usepackage[colorlinks]{hyperref}

\AtBeginDocument{%
  \hypersetup{
    citecolor=cyan,
    linkcolor=black,   
    urlcolor=cyan}}

\theoremstyle{plain}

\newtheorem{thm}{Theorem}[section]
\newtheorem{prop}[thm]{Proposition}
\newtheorem{lem}[thm]{Lemma}
\newtheorem{cor}[thm]{Corollary}

\theoremstyle{definition}

\theoremstyle{remark}
\newtheorem{rem}{Remark}[section]

\numberwithin{equation}{section}

\begin{document}

\title[]{Invariant subalgebras of the small $\mathcal{N}=4$ superconformal algebra } 

\author{Thomas Creutzig}
\author{Andrew R. Linshaw}
\author{Wolfgang Riedler}

\address{Department of Mathematical and Statistical Sciences, University of Alberta, Edmonton, Alberta  T6G 2G1, Canada.}
\email{creutzig@ualberta.ca}
\email{riedler@ualberta.ca}

\address{Department of Mathematics, University of Denver, Denver, CO 80208, U.S.A.}
\email{andrew.linshaw@du.edu}

\thanks{T. C. is supported by NSERC Discovery Grant RES0019997. A. L. is supported by Simons Foundation Grant 635650 and NSF Grant DMS-2001484. We thank the referees for their helpful comments and suggestions.}

\date{} 

\maketitle

\begin{abstract}
Various aspects of orbifolds and cosets of the small $\mathcal{N}=4$ superconformal algebra are studied. First, we determine minimal strong generators for generic and specific levels. As a corollary, we obtain the vertex algebra of global sections of the chiral de Rham complex on any complex Enriques surface. We also identify orbifolds of cosets of the small $\mathcal{N}=4$ superconformal algebra with $\text{Com}(V^{\ell}(\mathfrak{sl}_2), V^{\ell+1}(\mathfrak{sl}_2) \otimes \mathcal{W}_{-5/2}(\mathfrak{sl}_4, f_{\text{rect}}))$  and in addition at special levels with Grassmanian cosets and principal $\mathcal{W}$-algebras of type $A$ at degenerate admissible levels. These coincidences lead us to a novel level-rank duality involving Grassmannian supercosets.
\end{abstract}


\section{Introduction}\label{sec:Introduction}

The small $\mathcal{N}=4$ superconformal algebra $V^k(\mathfrak n_4)$ at level $k$ is a highly interesting family of vertex operator superalgebras. 
It is defined as the minimal $\mathcal{W}$-superalgebra of the universal affine vertex superalgebra $V^{-k-1}(\mathfrak{psl}_{2|2})$ of  $\mathfrak{psl}_{2|2}$ at level $-k-1$, and its affine subalgebra is the universal affine vertex algebra $V^k(\mathfrak{sl}_2)$ of $\mathfrak{sl}_2$ at level $k$. Moreover it has four dimension $3/2$ odd fields, hence the name $\mathcal{N}=4$ superconformal algebra. This algebra is a key ingredient in various problems of physics, as string theory on $K3$ surfaces \cite{Eguchi:1988af} and hence Mathieu moonshine \cite{Eguchi:2010ej}, the $AdS_3/CFT_2$ correspondence \cite{Maldacena:1997re} and as  chiral algebras of certain four-dimensional super Yang-Mills theories \cite{Bonetti:2018fqz}.
At level $k=1$, the simple quotient $L_1(\mathfrak{n}_4)$ of $V^1(\mathfrak{n}_4)$ appears as the algebra of global sections of the chiral de Rham complex of $K3$ surfaces \cite{2016arXiv160903688S}. More generally, for all integers $n\geq 1$ it is expected that $L_n(\mathfrak{n}_4)$ coincides with the algebra of global sections of the chiral de Rham complex of $2n$ complex-dimensional hyper-K\"ahler manifolds, see e.g. \cite{2008arXiv0806.1021H}.

$V^k(\mathfrak n_4)$ is exceptional in the sense that its group of outer automorphism is $SU(2)$, i.e. not a finite group.
It is surely important to gain a better understanding of vertex algebras related to $\mathcal{N}=4$ superconformal algebras. In this work, we study invariant subalgebras of $V^k(\mathfrak n_4)$, while in \cite{CFL} it is studied how small and also large $\mathcal{N}=4$ superconformal algebras can themselves be realized as cosets.  
 
 One of our main motivations is to determine the algebra of global sections of the chiral de Rham complex on Enriques surfaces.

\subsection{Chiral de Rham complex of Enriques surfaces}

It was shown in \cite{1999CMaPh.204..439M} that there exists a natural way to attach a vertex superalgebra valued sheaf $\Omega^{CdR}$, named the \emph{chiral de Rham} sheaf, to any smooth scheme $X$ of finite type over $\mathbb{C}$. 
It comes equipped with a square-zero differential $d$, and there exists a canonical embedding of the de Rham complex $(\Omega^{dR},d^{dR})\hookrightarrow (\Omega^{CdR},d)$ which is a quasi-isomorphism. The cohomology $H^*(X,\Omega^{CdR})$ with coefficients in $\Omega^{CdR}$ is always a vertex algebra, and the space of global sections $H^0(X,\Omega^{CdR})$ is a vertex subalgebra. For any $X$, $H^0(X,\Omega^{CdR})$ has a conformal structure, and if the first Chern class vanishes this extends to an $\mathcal{N}=2$ superconformal structure. Moreover, if $X$ is hyper-K\"ahler, this is further enhanced to an $\mathcal{N}=4$ structure on $H^0(X,\Omega^{CdR})$ with central charge $c=3dim_{\mathbb{C}}(X)$ \cite{2008arXiv0806.1021H}. Additional structures on $H^*(X,\Omega^{CdR})$ such as a chiral version of Poincar\'e duality are known to exist \cite{1999math......5008M}. It has been shown that this sheaf has connections to elliptic genera \cite{Borisov2000,Malikov2003}, mirror symmetry \cite{Borisov}, and physics \cite{Kapustin2005}.

Unfortunately, examples where the vertex algebra structure of $H^*(X,\Omega^{CdR})$ is understood are still lacking. Even the structure of $H^0(X,\Omega^{CdR})$ is only known in a few cases. For example, when $X$ is a $K3$ surface it was shown in \cite{2016arXiv160903688S} that $H^0(X,\Omega^{CdR})$ is isomorphic to the simple $\mathcal{N}=4$ vertex algebra $L_{1}(\mathfrak{n}_4)$, which has central charge c = 6.\footnote{See the preceeding work \cite{2013arXiv1312.7386S} for the construction restricting to Kummer surfaces.} (See also \cite{1999math......1065M} where $H^0(\mathbb{P}^N,\Omega^{CdR})$ was computed as a $\widehat{\mathfrak{sl}}_{N+1}$-module.) Recently some progress was made in \cite{2018arXiv180910320S} where for $X$ a compact Ricci-flat K\"ahler manifold $H^0(X,\Omega^{CdR})$ was shown to be isomorphic to a subspace of a $\beta\gamma-bc$-system that is invariant under the action of a certain Lie algebra. A first motivation for this article is to provide a further example to this list by concretely constructing the vertex algebra of global sections of $\Omega^{CdR}$ on any complex Enriques surface. The following is shown in Corollary \ref{cor:enriques} and Remark \ref{enriquesparafermions}
\begin{thm}
The vertex algebra of global sections of the chiral de Rham complex on a complex Enriques surface is of type $\mathcal{W}(1,2,4^4; \frac{3}{2}^2, \frac{7}{2}^2)$. This means that it has a minimal strong generating set consisting of one even field in weight $1$, one even field in weight $2$, four even fields in weight $4$, two odd fields in weight $\frac{3}{2}$, and two odd fields in weight $\frac{7}{2}$.
\end{thm}

Any complex Enriques surface $X$ can be constructed as the quotient of a $K3$ surface by a fixed-point free involution. Let $\iota$ be such an involution on a $K3$ surface $Y$. The action of the involution lifts to an action on the sheaf $\Omega^{CdR}$ and to its cohomology via automorphisms on the vertex algebra. A general construction of the chiral de Rham complex on orbifolds was given in \cite{2003math......7181F}. For $K3$ surfaces the automorphism on the vertex algebra induced by $\iota$ was already stated in \cite{2016arXiv160903688S}. The vertex algebra of global sections on Enriques surfaces is given by the fixed point set under this involution $H^0(X,\Omega^{CdR})=H^0(Y,\Omega^{CdR})^{\iota}$ (see Theorem 6.6 in \emph{op. cit.}).

\subsection{Invariant subalgebras of the small $\mathcal{N}=4$ superconformal algebra}
It is useful to place the problem of describing the $\mathbb{Z}/2\mathbb{Z}$-orbifold of $L_{1}(\mathfrak{n}_4)$ in the larger context of orbifolds of $V^k(\mathfrak{n}_4)$ under general reductive automorphisms groups, and cosets of $V^k(\mathfrak{n}_4)$ by general subalgebras. These problems have in fact been considered in \cite{ACKL17} for a general minimal $\mathcal{W}$-algebra. By Theorem 4.10 of \cite{ACKL17}, for any simple $\mathfrak{g}$ and any reductive automorphism group $G$, the coset $\mathcal{W}^k(\mathfrak{g}, f_{\text{min}})^G$ is strongly finitely generated for generic values of $k$. Additionally, the coset of $\mathcal{W}^k(\mathfrak{g}, f_{\text{min}})$ by its affine subalgebra is also strongly finitely generated for generic $k$; see Theorem 4.12 of \cite{ACKL17}. However, these results are nonconstructive, and it is useful to give explicit minimal strong generating sets in specific cases. In this paper, we will give minimal strong generating sets for $V^k(\mathfrak{n}_4)^{\mathbb{Z}/2\mathbb{Z}}$ and $V^k(\mathfrak{n}_4)^{U(1)}$ for generic values of $k$, and also determine the set of nongeneric values where our description fails. We also correct the description of $\mathcal{C}^k = \text{Com}(V^k(\mathfrak{sl}_2), V^k(\mathfrak{n}_4))$ that appeared in \cite{ACKL17}; it is in fact of type $\mathcal{W}(2,3^3,4,5^3,6)$. Furthermore, we show that  $\mathcal{C}^k$ has an additional action of $U(1)$ coming from the outer automorphism group of $V^k(\mathfrak{n}_4)$, and that $(\mathcal{C}^k)^{U(1)}$ is of type $\mathcal{W}(2,3,4,5,6,7,8)$. It arises as a quotient of the universal two-parameter $\mathcal{W}_{\infty}$-algebra constructed in \cite{L17}, and we identify it as a one-parameter VOA with another, seemingly unrelated coset, namely,
$$\text{Com}(V^{\ell}(\mathfrak{sl}_2), V^{\ell+1}(\mathfrak{sl}_2) \otimes \mathcal{W}_{-5/2}(\mathfrak{sl}_4, f_{\text{rect}})),$$ where $k$ and $\ell$ are related by $\displaystyle k = -\frac{\ell+1}{\ell+2}$. Finally, using this identification, we classify all isomorphisms between the simple quotient $(\mathcal{C}_k)^{U(1)}$ and various other structures such as type $A$ principal $\mathcal{W}$-algebras, generalized parafermion algebras, and cosets of type $A$ subregular $\mathcal{W}$-algebras.

\subsection{A new level-rank duality}
Coincidences between $(\mathcal{C}^k)^{U(1)}$ and principal $\mathcal W$-algebras of type $A$ appear at negative half-integer levels $k$. On the other hand these principal $\mathcal W$-algebras of type $A$ also appear as cosets by \cite{ACL19} and these cosets at positive integer level enjoy a nice level-rank duality with Grassmannian cosets \cite{OstrikSun}. We found an extension of this picture to negative integral levels and Grassmannian supercosets. 
Let $L_r(\mathfrak{sl}_m)$  denote the simple quotient of $V^r(\mathfrak{sl}_m)$ and $\mathcal H(1)$ the rank one Heisenberg vertex algebra. 
Our Theorem \ref{thm:levelrank} is 
\begin{thm} 
Let $r, n, m $ be positive integers. Then there exist vertex algebra extensions $A^{-n}(\mathfrak{sl}_{m})$ and $A^m(\mathfrak{sl}_{r|n})$ of homomorphic images $\widetilde V^{-n}(\mathfrak{sl}_{m})$ and $\widetilde V^m(\mathfrak{sl}_{r|n})$ of $V^{-n}(\mathfrak{sl}_{m})$ and $V^m(\mathfrak{sl}_{r|n})$
such that 
 the level-rank duality
\[
\text{Com}\left(V^{-n+r}(\mathfrak{sl}_m), A^{-n}(\mathfrak{sl}_{m}) \otimes L_r(\mathfrak{sl}_m) \right)\cong
\text{Com}\left(  V^{-m}(\mathfrak{sl}_n) \otimes L_m(\mathfrak{sl}_r) \otimes \mathcal H(1), A^m(\mathfrak{sl}_{r|n})\right)
\]
holds.
\end{thm}
It is natural to ask if the statement of the Theorem can be improved, i.e. one could ask for a  level-rank duality of the form
\[
\text{Com}\left(V^{-n+r}(\mathfrak{sl}_m), \widetilde  V^{-n}(\mathfrak{sl}_{m}) \otimes L_r(\mathfrak{sl}_m) \right)\overset{??}{\cong}
\text{Com}\left(  V^{-m}(\mathfrak{sl}_n) \otimes L_m(\mathfrak{sl}_r) \otimes \mathcal H(1), \widetilde  V^m(\mathfrak{sl}_{r|n})\right).
\]
Answering this question amounts to a better understanding of embeddings of the involved vertex superalgebras in the rank $rm$ $bc$-system times the rank $nm$ $\beta\gamma$-system, i.e. improving the results of \cite{LSS}.

\subsection{Organization}
The structure of the paper is as follows. In section \ref{sec:Background} some necessary background of the theory of vertex algebras is recalled and some notation is fixed. Then we quickly prove in section \ref{sec:3} that the group of outer automorphisms of $V^k(\mathfrak{n}_4)$ is $SL(2)$. 
Sections \ref{sec:Construction of the vertex algebra}  and \ref{sec:Construction of the cyclic orbifold} contain the construction of a few vertex algebras, the most important ones being $V^k(\mathfrak{n}_4)^{U(1)}$ and $V^k(\mathfrak{n}_4)^{\mathbb{Z}/2\mathbb{Z}}$. In section \ref{sec:Structure of the vertex algebra} some structure of the two central orbifolds from the previous sections are discussed and their simple quotients at all but finitely many levels $k$ are determined. 
In section \ref{sec:7} we determine cosets of $V^k(\mathfrak{n}_4)$ by its affine subalgebra at generic and specific levels. In particular, we identify its $U(1)$ orbifold with $\text{Com}(V^{\ell}(\mathfrak{sl}_2), V^{\ell+1}(\mathfrak{sl}_2) \otimes \mathcal{W}_{-5/2}(\mathfrak{sl}_4, f_{\text{rect}}))$  and in addition at special levels with Grassmanian cosets and principal $\mathcal{W}$-algebras of type $A$ at degenerate admissible levels.
The last section then discusses the new level-rank duality.
\\



\section{Background}\label{sec:Background}

\subsection{Vertex Algebras}
We assume that the reader is familiar with the basics of the theory of vertex algebras which has been discussed from multiple angles in the literature, some of which can be found in \cite{Borcherds:1983sq,frenkel2001vertex,kac2001vertex,1995q.alg.....1014L}. Given an element $a$ in a vertex algebra $\mathcal{A}$, the corresponding field will be denoted by $$a(z) = \sum_{n\in \mathbb{Z}} a_n z^{-n-1} \in \text{End}( \mathcal{A}) [[z, z^{-1}]].$$ The elements $a_n \in \text{End}(\mathcal{A})$ will be called the {\it modes} of $a(z)$. We often drop the formal variable $z$ and identify $a$ with $a(z)$ when no confusion can arise. Given $a,b \in \mathcal{A}$, the {\it normally ordered product} will be denoted by $:ab:$, and the $n^{\text{th}}$ product will be denoted by $a_{(n)} b$ for $n\in \mathbb{Z}$. Letting $\partial$ denote the translation operation $\frac{\partial}{\partial z}$, we have $:(\partial^n a) b: \ = n! a_{(-n-1)} b$, and in particular $:ab: \ = a_{(-1)} b$. For $n\geq 0$, the $n^{\text{th}}$ product gives the pole of order $n+1$ in the operator product expansion (OPE) formula: 
 \begin{equation} a(z)b(w) \sim \sum_{n=0}^{\infty} \frac{(a_{(n)} b) (w)}{(z-w)^{n+1}}.\end{equation} Here $\sim$ means equal modulo terms which are regular at $z = w$.

Let $\mathcal{V}$ be a vertex algebra. A set of fields $\mathcal{S}$ is said to \emph{strongly generate} $\mathcal{V}$ if it generates $\mathcal V$ as a differential algebra under normally ordered product. This means that every field of $\mathcal V$ can be written as a normally ordered polynomial in the fields in $\mathcal S$ and their iterated derivatives. If $\mathcal{S}$ is a minimal set with this property, and consists of $n_i$ even fields of weight $d_i$ for $i = 1,\dots, r$, and $m_j$ odd fields of weight $e_j$ for $j = 1,\dots, s$, we say that $\mathcal{V}$ if of type $$\mathcal{W}(d_1^{n_1},\dots, d_r^{n_r}; e_1^{m_1},\dots, e_s^{m_s}).$$

\subsection{Examples} In this subsection, we mention a few basic examples we shall need. First, let $\mathcal{H}(n)$ be the \emph{Heisenberg} vertex algebra of rank $n$. It is generated by even fields $h^i(z)$, $i=1,\dots,n$, which satisfy OPEs
 \[h^i(z)h^j(w)\sim \frac{\delta_{i,j}}{(z-w)^2}.\] The Virasoro element $L(z)=\frac{1}{2}\sum_{i=1}^n:h^i(z)h^i(z):$ has central charge $n$, and all generating fields are primary of weight 1. The automorphism group $Aut(\mathcal{H}(n))$ is isomorphic to the orthogonal group $O(n)$.

Let $\mathcal{A}(n)$ be the \emph{symplectic fermion} vertex algebra of rank $n$. It is generated by odd fields $e^i(z), f^i(z)$, $i=1,\dots,n$, which satisfy \[e^i(z)f^j(w)\sim \frac{\delta_{i,j}}{(z-w)^2},\qquad f^i(z)e^j(w)\sim -\frac{\delta_{i,j}}{(z-w)^2}.\] The Virasoro field $L(z)=-\sum_{i=1}^n:e^i(z)f^i(z):$ has central charge $-2n$, and all generating fields are primary of weight 1. The automorphism group $Aut(\mathcal{A}(n))$ is isomorphic to the symplectic group $Sp(2n)$.

Let $\mathcal{S}(n)$ denote the \emph{$\beta\gamma$-system} of rank $n$. It is generated by even fields $\beta^i(z), \gamma^i(z)$, $i=1,\dots,n$, which satisfy \[\beta^i(z)\gamma^j(w)\sim \frac{\delta_{i,j}}{(z-w)},\qquad \gamma^i(z)\beta^j(w)\sim -\frac{\delta_{i,j}}{(z-w)}.\] The Virasoro field  $L(z)=\frac{1}{2}\sum_{i=1}^n:\beta^i(z)\partial\gamma^i(z):-:\partial\beta^i(z)\gamma^i(z):$ has central charge $-n$, and all generating fields are primary of weight $\frac{1}{2}$. The automorphism group $Aut(\mathcal{S}(n))$ is isomorphic to the symplectic group $Sp(2n)$.

Let $\mathcal{E}(n)$ denote the \emph{$bc$-system} of rank $n$. It is generated by odd fields $b^i(z), c^i(z)$, $i=1,\dots,n$, which satisfy \[b^i(z)c^j(w)\sim \frac{\delta_{i,j}}{(z-w)},\qquad c^i(z) c^j(w)\sim \frac{\delta_{i,j}}{(z-w)}.\] The Virasoro field  $L(z)= -\frac{1}{2}\sum_{i=1}^n :b^i(z)\partial c^i(z):+:\partial b^i(z) c^i(z):$ has central charge $n$, and all generating fields are primary of weight $\frac{1}{2}$. The automorphism group $Aut(\mathcal{E}(n))$ is isomorphic to the symplectic group $O(2n)$.

Given a simple finite-dimensional Lie algebra $\mathfrak{g}$, the {\it universal affine vertex algebra} $V^k(\mathfrak{g})$ is generated by fields $\xi(z)$ which are linear in $\xi \in \mathfrak{g}$ and satisfy 
\begin{equation}
\xi(z) \eta(w) \sim \frac{k ( \xi, \eta)}{ (z-w)^{2}} + \frac{[\xi,\eta](w)}{(z-w)}.
\end{equation}
Here $(\cdot ,\cdot )$ denotes the normalized Killing form $\frac{1}{2h^{\vee}} \langle \cdot,\cdot \rangle$. For all $k\neq -h^{\vee}$, $V^k(\mathfrak{g})$ has a conformal vector
\begin{equation} \label{sugawara} L(z)  = \frac{1}{2(k+h^{\vee})} \sum_{i=1}^n :\xi_i(z) \xi'_i(z): \end{equation} of central charge $$c = \frac{k\ \text{dim}(\mathfrak{g})}{k+h^{\vee}}.$$ Here $\xi_i$ runs over a basis of $\mathfrak{g}$, and $\xi'_i$ is the dual basis with respect to $(\cdot,\cdot)$. As a module over $\widehat{\mathfrak{g}} = \mathfrak{g}[t,t^{-1}] \oplus \mathbb{C}$, $V^k(\mathfrak{g})$ is isomorphic to the vacuum $\widehat{\mathfrak{g}}$-module. If $k$ is generic, the vacuum module is irreducible and $V^k(\mathfrak{g})$ is a simple vertex algebra, but at special values of $k$ it is not simple and we denote by $L_k(\mathfrak{g})$ the simple quotient of $V^k(\mathfrak{g})$ by its maximal proper graded ideal.

\subsection{Filtrations} Let $\mathcal{V}$ be a vertex algebra. Suppose that $\mathcal{V}$ has a filtration 
\begin{equation} \label{eq:filtration} \mathcal{V}_{(0)}\subset\mathcal{V}_{(1)}\subset\cdots \quad\text{ where }\quad\mathcal{V}=\bigcup_{n=0}^{\infty}\mathcal{V}_{(n)}\end{equation} such that for any two elements $a\in\mathcal{V}_{(r)}$ and $b\in\mathcal{V}_{(s)}$ we have
\begin{equation*}
	a_{(n)} b\in
	\begin{cases}
		\mathcal{V}_{(r+s)} \hspace{22pt}\text{ for }\quad n<0 \\
		\mathcal{V}_{(r+s-1)} \quad\text{ for }\quad n\geq0.
	\end{cases}
\end{equation*}
Such filtrations were introduced by H. Li \cite{2002math......9310L}, and are known as {\it good increasing filtrations}. Setting $\mathcal{V}_{(-1)} = \{0\}$, the associated graded algebra $$gr(\mathcal{V}) = \bigoplus_{i\geq 0} \mathcal{V}_{(i)} / \mathcal{V}_{(i-1)}$$ is a graded, associative, (super-)commutative unital ring with a derivation where the multiplication is induced from the Wick product. We refer to such a ring as a \emph{$\partial$-ring}. A $\partial$-ring $\mathcal{A}$ is said to be generated by a set $\{a_i\in \mathcal{A}| \ i\in I\}$ if the set $\{\partial^na_i| \ i\in I, \ n\geq0\}$ generate $\mathcal{A}$ as a ring.

The following {\it reconstruction property} appears as Lemma 3.6 of \cite{2006math......5174L} 
\begin{lem} \label{THM Lian-Linshaw} Let $\mathcal{V}$ be a vertex algebra with a good increasing filtration. Suppose that $gr(\mathcal{V})$ is generated as a $\partial$-ring by a collection $\{a_i| \ i \in I\}$, where $a_i$ is homogeneous of degree $d_i$. Choose vertex operators $\alpha_i \in \mathcal{V}_{(d)}$ such that $\phi_{d_i} (\alpha_i) = a_i$. Then $\mathcal{V}$ is strongly generated by the collection $\{\alpha_i| \ i \in I\}$.
\end{lem}

In \cite{2015arXiv151109143A}, a slight modification called a \emph{weak increasing filtration} was introduced. This is a filtration \eqref{eq:filtration} on a vertex algebra $\mathcal{V}$ as above, where
$$a_{(n)} b\in \mathcal{V}_{(r+s)}\quad\text{ for }\quad n\in\mathbb{Z}.$$ Setting $\mathcal{V}_{(-1)}=\{0\}$ as above, the associated graded algebra $gr(\mathcal{V}) = \bigoplus_{i=0}^{\infty}\mathcal{V}_{(i)}/\mathcal{V}_{(i-1)}$ can be given a vertex algebra structure that is induced from $\mathcal{V}$, which need not be abelian. As above, a strong generating set for $gr(\mathcal{V})$ gives rise to a strong generating set for $\mathcal{V}$; see Lemma 4.1 of \cite{2015arXiv151109143A} and the discussion following it.

\subsection{Associative $G$-modules and orbifolds}
\label{sec: Kac-Radul isomorphism}
We will make use of an isomorphism that can be found in \cite{Kac1996}. For similar work see also \cite{doi:10.1155/S1073792896000566}. We recall here the result that will be used later on.

Let $A$ be an associative algebra over $\mathbb{C}$. Furthermore, let $G$ be a reductive group and let $\phi: G\rightarrow Aut(A)$ be a group homomorphism. If a $G$-module $V$ is simultaneously an $A$-module that is $G$-equivariant, i.e. \ \[g(av) = (\phi(g)a)(gv) \quad\text{ for }\quad g\in G,\hspace{2pt}a\in A,\hspace{2pt}v\in V,\] then we call $V$ a $(G,A)$-module. Let $A_0\subset A$ be the subalgebra that is invariant under the $G$-action. Assume that $V$ is a direct sum of at most countably many finite dimensional, irreducible $G$-modules. Let $M\subset V$ be such a module and set $V_M=\bigoplus_i M_i$ to be the sum of all $G$-submodules of $V$ which are isomorphic to $M$. The action of $A_0$ and $G$ commute which implies that $V_M$ is an $A_0$-module. In particular, the action of $A_0$ viewed as a map from $M$ given by $m\mapsto am$ for $a\in A_0$ shows that this map is a $G$-homomorphism. Choosing a 1-dimensional subspace $f\subset M$ fixes unique 1-dimensional subspaces $f_i$ in all other $G$-modules $M_i\subset V_M$ by Schur's Lemma. Schur's Lemma further implies that acting by any $a\in A_0$ on $f$ then either yields one of these subspaces $f_i\subset M_i$ or zero. Letting $V^M = \bigoplus_i f_i \subset V_M$, we have that $V^M$ is an $A_0$-module. This leads to an isomorphism $V_M\cong M\otimes V^M$ of $(G,A_0)$-modules. Therefore $$V\cong \bigoplus_{M\in\mathcal{S}} V_M \cong\bigoplus_{M\in\mathcal{S}}M\otimes V^M,$$ is a $(G,A_0)$-module isomorphism where $\mathcal{S}$ is the set of equivalence classes of simple $G$-modules. 
We need the following result, which is the same as Theorem 1.1 of \cite{Kac1996} and has the same proof, except that it is stated for $(G,A)$-modules rather than $(\mathfrak{g},A)$-modules.
\begin{thm} \label{thm:kacradul} Let $A$ be a semisimple $G$-module and let $V$ be a $(G,A)$-module such that 
\begin{enumerate}
\item $V$ is irreducible as an $A$-module.
\item $V$ is a direct sum of at most countably many finite-dimensional irreducible $G$-modules.
\end{enumerate}
Then each isotypic component $V_M$ is an irreducible $(G,A_0)$-module. Equivalently, each $V^M$ is irreducible as an $A_0$-module.
\end{thm}
Later on we will make use of this isomorphism in the context when $A$ is the algebra of modes which span the vector space of a vertex algebra and view it as a module over itself.

\subsection{Associated vertex algebra to $\mathfrak{n}_4$} 
In this article we denote the small $\mathcal{N}=4$ superconformal algebra by $\mathfrak{n}_4$ and its associated vertex superalgebra at level $k$ by $V^k(\mathfrak{n}_4)$. 
Note that here the level refers to the level of the affine subalgebra. $V^k(\mathfrak{n}_4)$ is actually the minimal $\mathcal{W}$-superalgebra of $\mathfrak{psl}_{2|2}$ at level $-k-1$. It has strong generators $J, J^{\pm}, T, G^{\pm,x}$ for $x = 1,2$, which satisfy the following OPE relations.
\begin{equation}
	\begin{split}
		\label{OPEs}
		T(z)T(w) &\sim \frac{\frac 12 c}{(z-w)^4}+\frac{2T(w)}{(z-w)^2}+\frac{\partial_w T(w)}{(z-w)} \\
		T(z)G^{\pm,x}(w) &\sim \frac{\frac 32 G^{\pm,x}(w)}{(z-w)^2}+\frac{\partial_w G^{\pm,x}(w)}{(z-w)} \\
		T(z)X(w) &\sim \frac{X(w)}{(z-w)^2}+\frac{\partial_wX(w)}{(z-w)} \hspace{50pt}\text{for } X\in\{J,J^{\pm}\}\\
		J(z)G^{\pm,x}(w) &\sim \pm\frac{G^{\pm,x}(w)}{(z-w)} \qquad \qquad
		J^{\pm}(z)G^{\mp,x}(w) \sim (-1)^x\frac{G^{\pm,x}(w)}{(z-w)} \\\
		J(z)J(w) &\sim \frac{2 k }{(z-w)^2} \qquad\qquad
		J(z)J^{\pm}(w) \sim \pm\frac{2J^{\pm}(w) }{(z-w)} \\
		J^{\pm}(z)J^{\mp}(w) &\sim \frac{k}{(z-w)^2} \pm \frac{J(w)}{(z-w)} \\
		G^{\pm,2}(z)G^{\pm,1}(w) &\sim \frac{2J^{\pm}(w)}{(z-w)^2} + \frac{\partial_wJ^{\pm}(w)}{(z-w)} \\
		G^{\mp,2}(z)G^{\pm,1}(w) &\sim \frac{\frac 13 c}{(z-w)^3} \mp \frac{J(w)}{(z-w)^2} + \frac{ T(w) \mp \frac 12  \partial_wJ(w)}{(z-w)} \\
	\end{split}
\end{equation}

From here onwards we adopt the following conventions:
\begin{align*}
	Q^+(z) 	&=G^{-,2}(z) 
	\quad&
	G^+(z) 	&=G^{+,2}(z) \\
	Q^-(z) 	&=G^{+,1}(z)
	\quad&
	G^-(z) 	&=G^{-,1}(z) 
\end{align*}

Note that $\{J, J^{\pm}\}$ and $\{J, T, G^{\pm}\}$ generate copies of $V^k(\mathfrak{sl}_2)$ and $V^k(\mathfrak{n}_2)$, respectively, inside $V^k(\mathfrak{n}_4)$.

\section{Automorphisms of $V^k(\mathfrak{n}_4)$}\label{sec:3} For later use, we determine the group of automorphisms $G$ of $V^k(\mathfrak{n}_4)$. First, there is the group of inner automorphisms $G_{\text{Inn}}$ obtained by exponentiating the zero modes of the weight $1$ fields; this is a copy of $SL_2$. Since the affine subalgebra $V^k(\mathfrak{sl}_2)$ has no outer automorphisms, the outer automorphism group $G_{\text{Out}}$ of   $V^k(\mathfrak{n}_4)$ is just the subgroup of $G$ consisting of automorphisms that fix $V^k(\mathfrak{sl}_2)$ pointwise. 

\begin{lem} $G_{\text{Out}}$ is a normal subgroup of $G$, and $G$ is the semidirect product $G_{\text{Out}} \rtimes G_{\text{Inn}}$.
\end{lem} 

\begin{proof} Clearly any inner automorphism which fixes $V^k(\mathfrak{sl}_2)$ is trivial, so $G_{\text{Inn}} \cap G_{\text{Out}}$ is trivial. Let $\omega \in G$. The restriction of $\omega$ to $V^k(\mathfrak{sl}_2)$ is an automorphism of $V^k(\mathfrak{sl}_2)$, which has only inner automorphisms, so there exists $\alpha \in G_{\text{Inn}}$ such that $\omega = \alpha$ on $V^k(\mathfrak{sl}_2)$. Letting $\beta = \alpha^{-1} \omega$ and $\gamma = \omega \alpha^{-1}$, it is easy to see that $\beta, \gamma \in G_{\text{Out}}$ and are the unique elements of $G_{\text{Out}}$ such that $\alpha \beta = \omega = \gamma \alpha$. The normality of  $G_{\text{Out}}$ is obvious from the definition, so the claim follows. 
\end{proof}

By weight considerations, $G_{\text{Out}}$ must act linearly on the weight $\frac{3}{2}$ subspace, which is the span of $\{G^{\pm}, Q^{\pm}\}$, and since it preserves OPEs between the weight $1$ fields and weight $\frac{3}{2}$ fields, it must preserve the two-dimensional spaces $\{ G^+, Q^- \}$ and $\{ Q^+,G^-\}$. Using the fact that $G_{\text{Out}}$ preserves OPEs between the weight $\frac{3}{2}$ fields, it is not difficult to check that $\omega \in G_{\text{Out}}$ must have the form 
\begin{equation} \begin{split} 
\omega(G^+) & =   a_0 G^+ + a_1 Q^-,\qquad \omega(Q^-) = b_0 G^+ + b_1 Q^-,
\\ \omega(Q^+) & =  a_0 Q^+ - a_1 G^-, \qquad \omega(G^-) = - b_0 Q^+ + b_1 G^- ,
\end{split} \end{equation} for constants $a_0, a_1, b_0, b_1 \in \mathbb{C}$, where $a_0 b_1 - a_1 b_0 = 1$. One can then identify $G_{\text{Out}}$ with $SL_2$ via $\omega \mapsto \left( \begin{matrix} a_0 & b_0 \\ a_1 & b_1  \end{matrix}\right) $. Moreover, it is easy to verify that $G_{\text{Out}}$ commutes with $G_{\text{Inn}}$. We obtain

\begin{thm} The automorphism group $G$ of $V^k(\mathfrak{n}_4)$ is isomorphic to $SL_2 \times SL_2$.
\end{thm}

\section{Construction of the vertex algebras}\label{sec:Construction of the vertex algebra}

In this section we construct the orbifolds of $V^k(\mathfrak{n}_4)$ and for this we use ideas from \cite{Al-Ali2017,2015arXiv151109143A, CLsub}. In order to describe the $\mathbb{Z}/N \mathbb{Z}$-orbifolds, it will be convenient to first describe the orbifold under a certain copy of $U(1)$ which contains all these cyclic groups as subgroups. We are interested in a $U(1)$-action that corresponds to a maximal torus of a diagonal $SL_2$ inside the automorphism group $SL_2 \times SL_2$.
Concretely, $\lambda \in U(1)$ acts as multiplication by $\lambda$ on $G^+$ and $J^+$, multiplication by $\lambda^{-1}$ on $G^-$ and $J^-$ and it leaves the remaining generators $Q^\pm, J$ and $T$ invariant. Note that restricting to the subgroup $\mathbb{Z}/2\mathbb{Z}\subseteq U(1)$ at level $k=1$ yields the automorphism on $H^0(X,\Omega^{CdR})$ on a $K3$ surface $X$ that is induced from a fixed point free involution as mentioned in the introduction.

Our strategy will be to first find minimal strong generating sets for the subalgebras $V^k(\mathfrak{sl}_2)^{U(1)}$ and V$^k(\mathfrak{n}_2)^{U(1)}$ of $V^k(\mathfrak{n}_4)^{U(1)}$. Then we will use these to construct a minimal strong generating set for $V^k(\mathfrak{n}_4)^{U(1)}$. Next, since $\mathbb{Z}/N \mathbb{Z}$ is a subgroup of $U(1)$, we may regard $V^k(\mathfrak{n}_4)^{\mathbb{Z}/N \mathbb{Z}}$ as a $V^k(\mathfrak{n}_4)^{U(1)}$-module. Using this module structure, we will give a generating set for $V^k(\mathfrak{n}_4)^{\mathbb{Z}/N \mathbb{Z}}$ for all $N$. In the case $N=2$, we will obtain a minimal strong generating set for $V^k(\mathfrak{n}_4)^{\mathbb{Z}/2 \mathbb{Z}}$. Finally, using the structure of singular vectors in the simple orbifold $L_k(\mathfrak{n}_4)^{\mathbb{Z}/2 \mathbb{Z}}$, we will obtain minimal strong generating sets for $L_k(\mathfrak{n}_4)^{\mathbb{Z}/2 \mathbb{Z}}$ at all values of $k$.

\subsection{Construction of $V^k(\mathfrak{sl}_2)^{U(1)}$.} 
We begin by determining the $U(1)$-orbifold of the vertex subalgebra $\mathcal{V}=V^k(\mathfrak{sl}_2)$ that is generated by $\{J,J^{\pm}\}$. The structure of this algebra is well-known (see Theorem 2.1 of \cite{DLWY10}), but we provide details for the benefit of the reader. First, $\mathcal{V}$ has a good increasing filtration 
$$\mathcal{V}_{(0)}\subset \mathcal{V}_{(1)}\subset \mathcal{V}_{(2)}\subset \cdots,\qquad \mathcal{V} = \bigcup_{r\geq 0} \mathcal{V}_{(r)},$$ where $\mathcal{V}_{(r)}$ is spanned by all iterated normally ordered monomials in the generators $J, J^{\pm}$ and their derivatives, of length at most $r$.
We defined $\mathcal{V}_{(-1)} = \{0\}$; then $$gr(\mathcal{V}) = \bigoplus_{r\geq 0} \mathcal{V}_{(r)} / \mathcal{V}_{(r-1)} \cong \mathbb{C}[J, \partial J, \partial^2 J,\dots, J^-,\partial J^-,\partial^2 J^-,\dots,  J^+,\partial J^+,\partial^2 J^+,\cdots]$$ as $\partial$-rings, that is, differential commutative rings.
Since $U(1)$ acts linearly on the generators, this action preserves the filtration. By restriction, we obtain a filtration on the orbifold $\mathcal{V}^{U(1)}$:
$$\mathcal{V}^{U(1)}_{(0)}\subset \mathcal{V}^{U(1)}_{(1)}\subset \mathcal{V}^{U(1)}_{(2)}\subset \cdots,\qquad \mathcal{V}^{U(1)}_{(i)} = \mathcal{V}^{U(1)}\cap \mathcal{V}_{(i)}.$$ Moreover, since $U(1)$ preserves the filtration on $\mathcal{V}$, it is apparent that 
$$gr(\mathcal{V})^{U(1)} \cong  \text{gr}\big(\mathcal{V}^{U(1)}\big).$$
The commutative ring $gr(\mathcal{V})^{U(1)}$ is spanned by monomials of the form $$\partial^{i_1} J \cdots \partial^{i_r} J \partial^{j_1} J^+ \cdots \partial^{j_s} J^+ \partial^{k_1} J^- \cdots \partial^{k_s} J^-.$$ These can all be generated by $\{\partial^i J\}_{i=0}^{\infty}$ together with the quadratics $\{u_{i,j} = \partial^i J^+ \partial^j J^-\}_{i,j = 0}^{\infty}$. From the point of view of $\partial$-rings, there is some redundancy in this generating set due to the Leibniz relations $\partial u_{i,j} = u_{i+1,j} + u_{i,j+1}$. It is easy to verify that sets $\{\partial^k u_{n,0}\}_{n,k = 0}^{\infty}$ and $\{u_{i,j}\}_{i,j = 0}^{\infty}$ span the same vector space, so that $\{J, u_{n,0}\}_{n=0}^{\infty}$ also generates $gr(\mathcal{V})^{U(1)}$ as a $\partial$-ring. In fact, it is easy to verify that this is a {\it minimal} generating set for $gr(\mathcal{V})^{U(1)}$ as a $\partial$-ring.

We now define the corresponding fields \begin{equation} \label{def:uij} U_{i,j}(z)= \ :\partial^iJ^+(z)\partial^jJ^-(z): \ \in \mathcal{V}^{U(1)}.\end{equation} 

\begin{lem}\label{LEM: V^U(1) - version 1} $\mathcal{V}^{U(1)}$ is strongly generated by the fields $\{J(z),U_{n,0}(z)\}_{n=0}^{\infty}$.
\end{lem}

\begin{proof} Since the corresponding elements $\{J, u_{n,0}\}_{n=0}^{\infty}$ generate $gr(\mathcal{V})^{U(1)}$ as a $\partial$-ring, this follows immediately from Lemma \ref{THM Lian-Linshaw}. 
\end{proof}

Even though $\{J, u_{n,0}\}_{n=0}^{\infty}$ is a minimal generating set for $gr(\mathcal{V})^{U(1)}$ as a $\partial$-ring, it turns out that a much smaller set than $\{J(z),U_{n,0}(z)\}_{n=0}^{\infty}$ is needed to strongly generate $\mathcal{V}^{U(1)}$ because of the existence of {\it decoupling relations} arising from quantum corrections to classical relations. Observe that the following relations hold in the commutative ring $gr(\mathcal{V})^{U(1)}$:
$$u_{a,b}u_{c,d} - u_{a,d}u_{c,b} = 0.$$ However, the corresponding normally ordered expression
 $$:U_{a,b}U_{c,d}:  - :U_{a,d}U_{c,b}:$$ in $\mathcal{V}^{U(1)}$ does {\it not} vanish due to the noncommutativity and nonassociativity of the Wick product, and it requires a correction. By weight considerations, we may write
 $$:U_{a,b}U_{c,d}:- :U_{a,d}U_{c,b}: \ = c_{N}U_{N,0}+f(J,U_{0,0},\dots,U_{N-1,0})$$ for some scalar $c_N$ where $a+b+c+d+2=N$ and $f$ is a sum of normally ordered monomials in the fields $J, U_{0,0},\dots, U_{N-1,0}$ and their derivatives. The Leibniz rule yields the identity $U_{a,b} = \sum_{i=0}^b(-1)^i\binom{b}{i}\partial^{b-i}U_{a+i,0}$, and similarly for $U_{c,d}$. This implies that the last equality can be rewritten 
 $$c_NU_{N,0} = g(J,U_{0,0},\dots,U_{N-1,0}),$$ where $g$ is again a normally ordered polynomial in $J, U_{0,0},\dots, U_{N-1,0}$ and their derivatives. This shows that $U_{N,0}$ can be written as a normally ordered polynomial in strong generators of lower conformal weight and their derivatives, provided that $c_N$ does not vanish. Following \cite{2015arXiv151109143A} we will refer to such an expression as a \emph{decoupling relation}.
 
What follows are two technical Lemmas which will serve as a preparation for the proof of Proposition \ref{PROP: V^U(1) - version 2}. For some $n\in\mathbb{N}$ we may write a given element $\omega\in\mathcal{V}_{(2)}^{U(1)}$ of conformal weight $n+2$ in the form
$$\omega = g_{\omega}(J,U_{0,0},\dots,U_{n,0}),$$ where $g_{\omega}$ is a sum of normally ordered monomials in $J,U_{0,0},\dots,U_{n,0}$, and their derivatives. Such an expression is not unique due to the existence of decoupling relations as well as different conventions for normal ordering. Let the coefficient of $U_{n-i,i}$ in $g_{\omega}$ be denoted by $c_{n,i}(\omega)$ and define 
$$c_n(\omega)=\sum_{i=0}^{n}(-1)^{i}c_{n,i}(\omega).$$

\begin{lem}
	\label{LEM: coefficient in decoupling relation is canonical}
	For any $\omega\in\mathcal{V}_{(2)}^{U(1)}$ of weight $n+5$ for $n\in\mathbb{N}$ the coefficient of $U_{n+3,0}$ appearing in $g_{\omega}$ is independent of all choices of normal orderings and is equal to $c_{n+3}(\omega)$.
\end{lem}
\begin{proof} The proof is analogous to the proof of Lemma 5.2 in \cite{2015arXiv151109143A}.
\end{proof}

\begin{rem} Note that we use strong generators of the form $U_{n,0}$ (\emph{cf. loc. cit.}). We may simply rewrite these by using the Leibniz rule as above with their difference being a total derivative. The factor $(-1)^n$ does not appear in our case because of the different choice in strong generators. 
\end{rem}

\begin{lem} 
	\label{LEM: - decoupling relations 1}
	Let $n\in\mathbb{N}_0$ and let $P_m$ denote a sum of normally ordered products of strong generators of $\mathcal{V}^{U(1)}$ of weight less than $m$ and their derivatives. We have
	\begin{align*}
		:U_{0,0}U_{1,n}: \ = & \left(\frac{2}{n+2} + \frac{k}{n+2}\right)U_{1,n+2} - \frac{1}{n+1} U_{2,n+1} + \\
		\makebox[2.5cm]{\hfill} & + \left(1+\frac{k}{3}\right)U_{3,n} + \frac{2}{(n+1)(n+2)(n+3)}U_{n+3,0} + P_{n+5}, \\
		:U_{0,n}U_{1,0}: \ = & \Big(\frac{2}{n+2} + \frac{k}{2} \Big)U_{1,n+2} - U_{2,n+1} + \\ 
		\makebox[2.5cm]{\hfill} & + \left\{\frac{2}{(n+1)(n+2)(n+3)} + (-1)^n\Big(\frac{1}{n+1} + \frac{k}{n+3}\Big)\right\}U_{n+3,0} + P_{n+5}.
	\end{align*}
\end{lem}

\begin{proof}
	The proof is a straightforward computation using only the definition of normal ordering and the commutation relations of $\widehat{\mathfrak{sl}}_2$ at level $k$.
\end{proof}

\begin{prop}
	\label{PROP: V^U(1) - version 2}
	For any non-vanishing level $k$ the vertex algebra $V^k(\mathfrak{sl}_2)^{U(1)}$ is of type $\mathcal{W}(1,2,3,4,5)$. A set of minimal strong generators is $\{J,U_{n,0}\}_{n=0}^3$.
\end{prop}

\begin{proof} We consider relations of the form \[\omega_n\stackrel{\mathclap{\normalfont\mbox{\scriptsize{def}}}}{=}\ :U_{0,0}U_{1,n}-U_{0,n}U_{1,0}:\] for $n\in\mathbb{N}$. The quadratic term vanishes due to the relation $u_{0,0} u_{1,n} - u_{0,n} u_{1,0} = 0$ in $gr(\mathcal{V})^{U(1)}$, so $\omega_n\in\mathcal{V}_{(2)}^{U(1)}$. Writing $\omega_n$ as a sum of normally ordered monimials in $J, U_{0,0},\dots, U_{n+3,0}$ and their derivatives, the coefficient of $U_{n+3,0}$ is canonical in the sense of Lemma \ref{LEM: coefficient in decoupling relation is canonical}. Due to Lemma \ref{LEM: - decoupling relations 1} it can be easily computed and equals \[c_{n+3}(\omega_n) = (-1)^{n+1}k\frac{n(n+5)}{6(n+2)(n+3)}.\] This shows existence of a decoupling relation for $U_{n+4,0}$ for all $n\in\mathbb{N}_0$ at any non-vanishing level $k$. Hence, $V^k(\mathfrak{sl}_2)^{U(1)}$ is strongly generated by $\{J,U_{n,0}\}_{n=0}^3$ for $k\neq 0$. Since there are no relations in $gr(\mathcal{V})^{U(1)}$ among $J, u_{n,0}$ for $n \leq 3$, there can be no further decoupling relation among the above strong generators.
\end{proof}

\subsection{Construction of $V^k(\mathfrak{n}_2)^{U(1)}$.}  The last Proposition can be used to prove a similar statement about $V^k(\mathfrak{n}_2)^{U(1)}$. It will be convenient to introduce the following fields \[V_{a,b}\stackrel{\mathclap{\normalfont\mbox{\scriptsize{def}}}}{=}\ :\partial^aG^{+}\partial^bG^{-}:.\] 
\begin{prop}
	\label{U(1)-orbifold of N=2 vertex algebra}
	For any level $k\neq0,-2$ the vertex algebra $V^k(\mathfrak{n}_2)^{U(1)}$ is of type $\mathcal{W}(1,2,3,4,5)$. A set of minimal strong generators is $\{J,T,V_{n,0}\}_{n=0}^2$.
\end{prop}
\begin{proof}
	Denote the standard strong generators of $V^k(\mathfrak{sl}_2)$ by $\{H,X^{\pm}\}$ and let $\mathcal{E}$ be the $bc$-system of rank 1. Considering the tensor product $V^k(\mathfrak{sl}_2)\otimes\mathcal{E}$, we will abuse notation and denote the strong generators by the same symbols. Let $K=\frac{1}{2}H-:bc:$. The zero mode $K_0$ integrates to a $U(1)$-action on $V^k(\mathfrak{sl}_2)\otimes\mathcal{E}$. Lemma 8.6 in \cite{2014arXiv1407.8512C} shows that if $k\neq-2$ then 
	\begin{equation}
	\label{Isomorphism of V(n2) and V(sl2) x E}
		V^{\ell}(\mathfrak{n}_2)\cong \text{Com}(\mathcal{H}_1,V^k(\mathfrak{sl}_2)\otimes\mathcal{E})
	\end{equation}
	 with central charge $\frac{3k}{k+2}$ where $\mathcal{H}_1$ is the Heisenberg vertex algebra generated by $K$. The zero mode $H_0$ integrates to a $U(1)$-action as well. Denote the group associated to the zero mode of $K$ ($H$) by $G_1$ ($G_2$) and let $\mathcal{H}_2$ be the Heisenberg vertex algebra generated by $H$. 
	\begin{align*}
		\mathcal{H}_1\otimes V^{\ell}(\mathfrak{n}_2)^{G_2} &\cong \left((V^k(\mathfrak{sl}_2)\otimes\mathcal{E})^{G_2}\right)^{G_1} \\
		&\cong (\mathcal{H}_2\otimes \text{Com}(\mathcal{H}_2,V^k(\mathfrak{sl}_2)\otimes\mathcal{E}))^{G_1} \\
		&\cong \mathcal{H}_2\otimes \text{Com}(\mathcal{H}_2,V^k(\mathfrak{sl}_2))\otimes\mathcal{E}^{G_1} 
	\end{align*}
	The commutativity of the two group actions was used in the first equality. All $G_1$-invariant fields in $\mathcal{E}$ can be strongly generated by the fields $:\partial^ib\partial^jc:$ for $i,j\geq0$. By the action of the derivation a set of strong generators is given by $\{:\partial^ibc:\}_{i=0}^{\infty}$. It is easy to show that the equality \[:(:\partial^nbc:)(:bc:): \ = \frac{n+2}{n+1}:\partial^{n+1}bc:+\partial\omega\] holds for $n\geq0$ where $\omega$ is a linear combination of the fields $\partial^{n-i}:\partial^ibc:$ for $i=0,\dots,n$. This implies that $:bc:$ strongly generates $\mathcal{E}^{G_1}$. From the above isomorphism one can deduce
	\begin{align*}
		V^{\ell}(\mathfrak{n}_2)^{G_2} &\cong \text{Com}(\mathcal{H}_1,\mathcal{H}_2\otimes \text{Com}(\mathcal{H}_2,V^k(\mathfrak{sl}_2))\otimes\mathcal{E}^{G_1}).
	\end{align*}
	It now follows from Proposition \ref{PROP: V^U(1) - version 2} that the only strong generator of weight 1 is an element of the commutant $\text{Com}(\mathcal{H}_1,\mathcal{H}_2\otimes\mathcal{E}^{G_1})$ which is isomorphic to a Heisenberg vertex algebra. Hence, $V^k(\mathfrak{n}_2)^{G_2}$ is of type $\mathcal{W}(1,2,3,4,5)$. The Heisenberg and the Virasoro field of $V^k(\mathfrak{n}_2)$ are elements of the kernel of $H_0$. The isomorphism in (\ref{Isomorphism of V(n2) and V(sl2) x E}) for the weight $\frac{3}{2}$ fields is given by 
$$:X^{+}b: \ \mapsto \delta^+ G^+,\qquad :X^- c: \ \mapsto \delta^- G^-,\ \text{where} \ \delta^+\delta^-=2+k.$$
From the action of the zero mode $H_0$ it is immediate that a set of strong generators of $V^k(\mathfrak{n}_2)^{G_2}$ of weight 3, 4 and 5 can be given by $\{V_{n,0}\}_{n=0}^2$.
\end{proof}

\subsection{Construction of $V^k(\mathfrak{n}_4)^{U(1)}$.} We are now ready to find a minimal strong generating set for $V^k(\mathfrak{n}_4)^{U(1)}$, where $\lambda \in U(1)$ acts as multiplication by $\lambda$ on $G^+$ and $J^+$, multiplication by $\lambda^{-1}$ on $G^-$ and $J^-$, and fixes the remaining generators $Q^\pm, J$ and $T$. First, we define a filtration on $V^k(\mathfrak{n}_4)$,
$$V^k(\mathfrak{n}_4)_{(0)} \subset V^k(\mathfrak{n}_4)_{(1)} \subset \cdots, $$ where $V^k(\mathfrak{n}_4)_{(n)}$ is spanned by all normally ordered monomials in the generators $J, J^{\pm}, T, G^{\pm}, Q^{\pm}$ and their derivatives, such that the total number of fields $G^{\pm}, Q^{\pm}$ and their derivatives, is at most $n$. It is easily verified from the OPE algebra that this is a weak increasing filtration. Setting $V^k(\mathfrak{n}_4)_{(-1)} = \{0\}$,  let $$\mathcal{W} = gr(V^k(\mathfrak{n}_4))= \bigoplus_{n\geq 0} V^k(\mathfrak{n}_4)_{(n)}.$$ The OPEs among $G^{\pm}, Q^{\pm}$ are now regular in $\mathcal{W}$ but the remaining OPEs are unchanged. 

Next, $\mathcal{W}$ has a good increasing filtration where $\mathcal{W}_{(n)}$ is spanned by all normally ordered monomials in the generators $J, J^{\pm}, T,G^{\pm}, Q^{\pm}$ and their derivatives of total length at most $n$. Clearly $\text{gr}(\mathcal{W})$ is just the differential polynomial superalgebra on these generators and their derivatives. It is evident that the action of $U(1)$ preserves the weak filtration on $V^k(\mathfrak{n}_4)$ as well as the good filtration on $\mathcal{W} = gr(V^k(\mathfrak{n}_4))$. Restricting the filtrations to the $U(1)$ orbifolds, it follows that 
$$gr(V^k(\mathfrak{n}_4)^{U(1)}) \cong \mathcal{W}^{U(1)},\qquad gr(\mathcal{W}^{U(1)}) \cong gr(\mathcal{W})^{U(1)}.$$

The commutative ring $gr(\mathcal{W})^{U(1)}$ is clearly generated as a $\partial$-ring by the even elements $J, T, U_{n,0}, V_{n,0}$ together with the odd elements
$$Q^{\pm}, \qquad a_{n,0} = \partial^n J^+G^{-} ,\qquad b_{n,0} = \partial^nJ^-G^{+}, \qquad n\geq 0.$$ It then follows from Lemma \ref{THM Lian-Linshaw} and the analogous statement for weak filtrations, that the corresponding fields in $V^k(\mathfrak{n}_4)^{U(1)}$, namely, the even fields $J, T, U_{n,0}, V_{m,0}$ together with the odd fields
$$Q^{\pm}, \qquad A_{n,0} = \ :\partial^nJ^+G^{-}:,\qquad  B_{n,0} = \ \ :\partial^nJ^-G^{+}:,\qquad n\geq 0,$$
are a strong generating set for $V^k(\mathfrak{n}_4)^{U(1)}$.

In particular, note that the vertex algebras from Proposition \ref{PROP: V^U(1) - version 2} and \ref{U(1)-orbifold of N=2 vertex algebra} are vertex subalgebras of $V^k(\mathfrak{n}_4)^{U(1)}$. We already have shown that $U_{n,0}$ and $V_{m,0}$ for $n\geq 4$ and $m\geq 3$, are unnecessary. 
Even though $U_{3,0}$ was needed in $V^k(\mathfrak{sl}_2)^{U(1)}$, it turns out to be unnecessary in $V^k(\mathfrak{n}_4)^{U(1)}$ since there exists a decoupling relation for all non-vanishing levels $k$; see (\ref{Appendix: Decoupling relation for U_{3,0}}). In order to find a minimal strong generating set for $V^k(\mathfrak{n}_4)^{U(1)}$, we shall construct decoupling relations for $A_{n,0}$ and $B_{n,0}$ for all $n\geq 3$. The proofs of the following two Lemmas are straightforward computations and will be omitted.

\begin{lem}
	Let $P_{m}$ be a sum of normally ordered products of weight $m$ such that each summand includes the field $\partial^i T$ or $\partial^j J$, and at most one field $\partial^k V_{l,0}$ for some $i,j,k,l,m\in\mathbb{N}_0$. Let $n\in\mathbb{N}_0$. 
	\begin{align*}
		\frac{k(6+n)}{3(3+n)}V_{n+3,0} =& :V_{0,0}V_{n,0}: + k\partial V_{n+2,0} - 
		\left(\frac{1}{2+2n}+k\right)\partial^2 V_{n+1,0} + \frac{k}{3}\partial^3 V_{n,0} + P_{n+6}
	\end{align*}
\end{lem}

\begin{lem}
	Let $n\in\mathbb{N}_0$ and let $P_m$ be a sum of normally ordered products of strong generators of $V^k(\mathfrak{n}_4)^{U(1)}$ of weight less than $m$ and their derivatives.
	\begin{align*}
	\frac{nk}{2(n+2)}A_{n+2,0} + P_{n+\frac{9}{2}} = & :U_{n,0}A_{0,0}: - :U_{0,0}A_{n,0}: + :(U_{n+1,0}-\partial U_{n,0})Q^-: \\
	 	\makebox[75pt]{\hfill} & - \frac{1}{n+1}\sum_{i=0}^{n+1}(-1)^i\binom{n}{i}:\partial^iU_{n+1-i,0}Q^-: \\
	\frac{nk}{2(n+2)}B_{n+2,0} + P_{n+\frac{9}{2}} = & \sum_{i=0}^n(-1)^i\binom{n}{i}\left\{:\partial^iU_{n-i,0}B_{0,0}: + :\partial^iU_{n+1-i,0}Q^+:\right\} \\ 
		\makebox[75pt]{\hfill} & - \frac{1}{n+1}:U_{n+1,0}Q^+: + (-1)^{n+1}:U_0B_{n,0}: 
	\end{align*}
\end{lem}

The last Lemma shows the existence of decoupling relations for the fields in the set $\{A_{i,0},B_{i,0}\}_{i=3}^{\infty}$. As before it can be checked directly that there are no decoupling relations for the remaining strong generators. Hence, this yields the following. 

\begin{thm}
\label{THM: Minimal strong generators of U(1)-orbifold of V^k(n_4)}
	For any level $k\neq-2,0$ the vertex algebra $V^k(\mathfrak{n}_4)^{U(1)}$ is of type $\mathcal{W}(1,2^2,3^2,4^2,5; \frac{3}{2}^2, \frac{5}{2}^2, \frac{7}{2}^2,\frac{9}{2}^2)$. A set of minimal strong generators is given by \newline $\{J,Q^{\pm},T,U_{i,0},A_{i,0},B_{i,0},V_{i,0}\}_{i=0}^2$.
\end{thm}

\section{Construction of the cyclic orbifold}\label{sec:Construction of the cyclic orbifold}

Let $\mathcal{U}(\mathfrak{n}_4)^{G}$ be the universal enveloping algebra of $\mathfrak{n}_4$ that is invariant under the group $G$. From here onwards we will restrict the group to be cyclic $G\cong\mathbb{Z}/N\mathbb{Z}$. As shown in section \ref{sec: Kac-Radul isomorphism}, there are isomorphisms \[\bigoplus_{i=0}^{N-1}\mathbb{C}_{i}\otimes W_{i}\cong V^k(\mathfrak{n}_4)\cong\bigoplus_{i=-\infty}^{\infty}\mathbb{C}_{i}\otimes V_{i}\] as a $(\mathbb{Z}/N\mathbb{Z},\mathcal{U}(\mathfrak{n}_4)^{\mathbb{Z}/N\mathbb{Z}})$-module and as a $(U(1),\mathcal{U}(\mathfrak{n}_4)^{U(1)})$-module, respectively. By restriction this leads to an isomorphism \[W_a\cong \bigoplus_{k=-\infty}^{\infty}V_{kN+a}\] for $a=0,\dots,N-1$ as a $\mathcal{U}(\mathfrak{n}_4)^{\mathbb{Z}/N\mathbb{Z}}$-module. 

It will be convenient to define the following fields \[Y(\sigma^{(i)\pm}_{a_1,\dots,a_{jN}},z)=\Sigma^{(i)\pm}_{a_1,\dots,a_{jN}} \stackrel{\mathclap{\normalfont\mbox{\scriptsize{def}}}}{=} :\partial^{a_1}J^{\pm}\cdots\partial^{a_{jN-i}}J^{\pm}\partial^{a_{jN-i+1}}G^{\pm}\cdots\partial^{a_{jN}}G^{\pm}:\] for $i=0,\dots,jN$ and $a_i\in\mathbb{N}_0$.
\begin{lem}
\label{LEM: strong generators of cyclic orbifold of V^k(n_4)}
	The vertex algebra $V^k(\mathfrak{n}_4)^{\mathbb{Z}/N\mathbb{Z}}$ is strongly generated by the strong generators of $V^k(\mathfrak{n}_4)^{U(1)}$ and the fields $\Sigma^{(i)\pm}_{a_1,\dots,a_N}$ for $i=0,\dots,N$ and $a_1,\dots,a_N\geq0$.
\end{lem}
\begin{proof}
	Let $V_{\pm} = \bigoplus_{j\in\mathbb{N}} V_{\pm jN}$ and let $i=0,\dots,N$. Each vector space $V_{jN}$ for $j\in\mathbb{Z}\backslash\{0\}$ is a $\mathcal{U}(\mathfrak{n}_4)^{U(1)}$-module and generated by the vectors in the set $\{\sigma^{(i)+}_{a_1,\dots,a_{jN}}\}_{a_1,\dots,a_{jN}=0}^{\infty}$. Note that $V_0 \oplus V_{+}$ is a vertex subalgebra of $V^k(\mathfrak{n}_4)^{\mathbb{Z}/N\mathbb{Z}}$. Due to commutativity of the fields in the set $\{\Sigma^{(i)+}_{a_1,\dots,a_N}\}_{i=0}^N$, $V_0 \oplus V_{+}$ is strongly generated by these fields and the strong generators of $V_0$. By the same argument the vertex subalgebra $V_0 \oplus V_{-}$ is strongly generated by the fields in the set $\{\Sigma^{(i)-}_{a_1,\dots,a_N}\}_{i=0}^{N}$ and the strong generators of $V_0$. Observe that the fields that appear in the OPEs between the strong generators in $V_{N}$ and $V_{-N}$ are necessarily elements of $V_0$. This proves the proposition. 
\end{proof}

\begin{rem} If $k$ is generic, $V^k(\mathfrak{n}_4)$ is simple, and hence $V^k(\mathfrak{n}_4)^{\mathbb{Z}/2\mathbb{Z}}$ and $V^k(\mathfrak{n}_4)^{U(1)}$ are both simple as well \cite{doi:10.1155/S1073792896000566}. It follows from Theorem \ref{thm:kacradul} that each of the $V^k(\mathfrak{n}_4)^{U(1)}$-modules $V_{i}$ is irreducible. In particular, the $V^k(\mathfrak{n}_4)^{U(1)}$-submodule of $V_N$ generated by $\Sigma^{(0)+}_{0,\dots,0} = \ :(J^+)^N:$, must be all of $V_N$. Similarly, $V_{-N}$ is generated by $\Sigma^{(0)-}_{0,\dots,0} = \ :(J^+)^N:$ over $V^k(\mathfrak{n}_4)^{U(1)}$. This implies that $V^k(\mathfrak{n}_4)^{\mathbb{Z}/2\mathbb{Z}}$ is generated (but not strongly) by $:(J^+)^N:$  and $:(J^-)^N:$ together with $V^k(\mathfrak{n}_4)^{U(1)}$. \end{rem}

In order to construct a more economical strong generating set for $V^k(\mathfrak{n}_4)^{\mathbb{Z}/2\mathbb{Z}}$, we first observe that some of the fields $\Sigma^{(i)+}_{a_1,\dots,a_{jN}}$ are superfluous as strong generators due to the action of the translation operator $\partial$. The generating function of the fields $\Sigma^{(l)\pm}_{a_1,\dots,a_N}$ for fixed $l$ with respect to weight will show which of these can be neglected as strong generators. It can be obtained by a simple counting argument. 
\begin{lem} 
\label{LEM: Generating function}
	The generating function for the number of fields $\Sigma^{(l)\pm}_{a_1,\dots,a_N}$ with respect to conformal weight is \[q^{N+\frac{l^2}{2}}\prod_{i=1}^{N-l}\frac{1}{1-q^i}\prod_{j=1}^{l}\frac{1}{1-q^j}.\]
\end{lem}
\begin{proof}
	The weight of the field $\Sigma^{(l)\pm}_{a_1,\dots,a_N}$ equals $N+\frac{1}{2}l+a+b$ where $a=\sum_{i=1}^{N-l}a_i$ and $b=\sum_{i=N-l+1}^Na_i$. Due to (anti-)commutativity of the fields in the set $\{\partial^iJ^+,\partial^iG^+\}_{i=0}^{\infty}$ the number of fields at a given weight is determined by the number of partitions of $a$ with at most $N-l$ parts and the number of partitions of $b$ with exactly $l$ and with exactly $l-1$ parts such that all summands are distinct in both cases. The latter condition exists since at most one of the coefficients in the set $\{a_{N-i+1},\dots,a_N\}$ can be zero, otherwise $\Sigma^{(i)\pm}_{a_1,\dots,a_N}$ vanishes.
\end{proof}
Let $V_{(l)}$ be the subspace of $V^k(\mathfrak{n}_4)$ spanned by the vectors in the set $\{\sigma^{(l)\pm}_{a_1,\dots,a_N}\}_{a_1,\dots,a_N=0}^{\infty}$. It is obvious that the translation operator $\partial$ acts on $V_{(l)}$ and increases the conformal weight by 1. Since we are looking for strong generators, we can eliminate any elements in the space $\partial V_{(l)}$. Dropping vectors in the set $\{\sigma^{(l)\pm}_{a_1,\dots,a_N}\}_{a_1,\dots,a_N=0}^{\infty}$ which decouple for this reason amounts to multiplying the generating function from Lemma \ref{LEM: Generating function} by $(1-q)$. 

Next, it follows from Theorem 4.10 \cite{ACKL17} that for generic $k$, only {\it finitely many} of the fields $\Sigma^{(i)\pm}_{a_1,\dots,a_N}$ are needed to strongly generate $V^k(\mathfrak{n}_4)^{\mathbb{Z}/2\mathbb{Z}}$. This ultimately boils down to the fact that each of the modules $V_i$ is $C_1$-cofinite as a module over $V^k(\mathfrak{n}_4)^{U(1)}$; a proof can be given along the lines of Theorems 6.1 and 6.3 of \cite{CKLR}.

It is a difficult problem to determine the {\it minimal} set of fields of the form $\Sigma^{(i)\pm}_{a_1,\dots,a_N}$ that are needed to strongly $V^k(\mathfrak{n}_4)^{\mathbb{Z}/2\mathbb{Z}}$ for all $N$, but we shall carry this out in the case $N=2$. First, multiplying the generating function from the previous Lemma by $(1-q)$ yields
\begin{equation*}
	l=0:\quad \sum_{i=0}^{\infty}q^{2n+2}, \qquad l=1:\quad \sum_{i=0}^{\infty}q^{n+\frac{5}{2}} , \qquad l=2:\quad \sum_{i=0}^{\infty}q^{2n+4}.
\end{equation*}
Thus, Lemma \ref{LEM: strong generators of cyclic orbifold of V^k(n_4)} can be improved upon since a set of strong generators of $V^k(\mathfrak{n}_4)^{\mathbb{Z}/2\mathbb{Z}}$ is given by the strong generators of $V^{k}(\mathfrak{n}_4)^{U(1)}$ together with the fields in the set \newline $\{\Sigma^{(0)\pm}_{2n,0},\Sigma^{(1)\pm}_{n,0},\Sigma^{(2)\pm}_{2n+1,0}\}_{n=0}^{\infty}$. The following Lemma will set up the proof of Theorem \ref{THM: Main Theorem. minimal strong generators of cyclic orbifold Z/2Z}.
\begin{lem}
\label{LEM: Decoupling relations for the strong generators of the cyclic orbifold that are not in the U(1) orbifold.}
	Let $n\in\mathbb{N}$ and $c^{(i)}_n,d^{(i)}_n\in\mathbb{Q}$ for $i=0,1,2$.
	\begin{align*}
		:U_{2n,0}\Sigma^{(0)+}_{0,0}-U_{0,0}\Sigma^{(0)+}_{2n,0}:     				  &= p^{(0)}_n(k) \hspace{2pt} \Sigma^{(0)+}_{2n+2,0} + \sum_{i=0}^{2n}c^{(0)}_i:\partial^iH\Sigma^{(0)+}_{2n-i,0}: \\
		:U_{2n-1,0}\Sigma^{(0)+}_{1,0}-U_{0,0}\Sigma^{(0)+}_{2n-1,1}: 				  &= q^{(0)}_n(k) \hspace{2pt} \Sigma^{(0)+}_{2n+2,0} + \sum_{i=0}^{2n}d^{(0)}_i:\partial^iH\Sigma^{(0)+}_{2n-i,0}: \\
		:\Sigma^{(0)+}_{n,0}B_{0,0}-U_{n,0}\Sigma^{(1)+}_{0,0}+\Sigma^{(0)+}_{n+1,0}Q^+: &= p^{(1)}_n(k) \hspace{2pt} \Sigma^{(1)+}_{n+2,0} \hspace{4pt}		 + \sum_{i=0}^{n+1}c^{(1)}_i:\partial^iH\Sigma^{(1)+}_{n+1-i,0}:\\	
		:\Sigma^{(0)+}_{n,0}B_{0,0}-U_{0,0}\Sigma^{(1)+}_{n,0}+\Sigma^{(0)+}_{n,1}Q^+:	  &= q^{(1)}_n(k) \hspace{2pt} \Sigma^{(1)+}_{n+2,0} \hspace{4pt} + \sum_{i=0}^{n+1}d^{(1)}_i:\partial^iH\Sigma^{(1)+}_{n+1-i,0}:\\
		:B_{0,0}\Sigma^{(1)+}_{2n-1,0} - \Sigma^{(1)+}_{2n-1,1}Q^+:	  				  &= p^{(2)}_n(k) \hspace{2pt} \Sigma^{(2)+}_{2n+1,0} + \sum_{i=0}^{2n}c^{(2)}_i:\partial^iH\Sigma^{(2)+}_{2n-i,0}:\\
		\end{align*}
	\begin{align*}
		:B_{2n-1,0}\Sigma^{(1)+}_{0,0} +\frac{1}{2n}\Sigma^{(1)+}_{0,2n}Q^+:  		  &= q^{(2)}_n(k) \hspace{2pt} \Sigma^{(2)+}_{2n+1,0} + \sum_{i=0}^{2n}d^{(2)}_i:\partial^iH\Sigma^{(2)+}_{2n-i,0}:
	\end{align*}
	\scriptsize
	\begin{align*}
		p^{(0)}_n(k) &= -\frac{8n}{2n+1}+\frac{n}{2n+2}k 
	  	&q^{(0)}_n(k) =& \frac{4}{3}+\frac{7+2n}{6(2n+1)}k \\
	  	p^{(1)}_n(k) &= \frac{2}{n+2}\left[\left(n-\frac{1+(-1)^n}{n+1}\right)+(-1)^n\frac{k}{2}\right]
	  	&q^{(1)}_n(k) =& \frac{1}{n+2}\left[\frac{(-1)}{n+1}[4n+5+(-1)^{n+1}(2n+1)]+\frac{k}{2}[n+2(-1)^n]\right] \\
	  	p^{(2)}_n(k) &= \frac{2+k}{2n+1}
	  	&q^{(2)}_n(k) =& \frac{2-k}{2n+1}
	\end{align*}
\normalsize
\end{lem}

\begin{proof}
	Computing the left hand side of each of these equations is straightforward and leads directly to the right hand side by using only the definition of normal ordering and the commutation relations of $\mathfrak{n}_4$ which are equivalent to the OPEs as stated in (\ref{OPEs}).
\end{proof}

\begin{thm}
\label{THM: Main Theorem. minimal strong generators of cyclic orbifold Z/2Z}
	Let $k\neq -2,0,4,16$. The vertex algebra $V^k(\mathfrak{n}_4)^{\mathbb{Z}/2\mathbb{Z}}$ is of type $\mathcal{W}(1,2^4,3^2,4^5; \frac{3}{2}^2, \frac{5}{2}^4, \frac{7}{2}^4)$. A set of minimal strong generators is \[\mathcal{S}=\{H,Q^{\pm},T,U_{i,0},A_{i,0},B_{i,0},V_{i,0},\Sigma^{(i)\pm}_{0,0},\Sigma^{(0)\pm}_{2,0},\Sigma^{(1)\pm}_{1,0},\Sigma^{(2)\pm}_{1,0}\}_{i=0}^1.\] At level $k=4$ a set of minimal strong generators is $\mathcal{S}\cup\{\Sigma^{(1)\pm}_{2,0}\}$. At level $k=16$ a set of minimal strong generators is $\mathcal{S}\cup\{U_{2,0},A_{2,0},B_{2,0}\}$.
\end{thm}
\begin{proof}
	By Lemma \ref{LEM: strong generators of cyclic orbifold of V^k(n_4)} and the discussion thereafter the vertex algebra $V^k(\mathfrak{n}_4)^{\mathbb{Z}/2\mathbb{Z}}$ is strongly generated by the strong generators of $V^k(\mathfrak{n}_4)^{U(1)}$ and the fields in the set $\{\Sigma^{(0)\pm}_{2n,0},\Sigma^{(1)\pm}_{n,0},\Sigma^{(2)\pm}_{2n+1,0}\}_{n=0}^{\infty}$. Observe in Lemma \ref{LEM: Decoupling relations for the strong generators of the cyclic orbifold that are not in the U(1) orbifold.} that for any $i\in\{0,1,2\}$ the roots of $p_n^{(i)}$ and $q_n^{(i)}$ are distinct for all $n\in\mathbb{N}$. It follows that the fields $\{\Sigma^{(0)+}_{2n+4,0},\Sigma^{(1)+}_{n+3,0},\Sigma^{(2)+}_{2n+3,0}\}_{n=0}^{\infty}$ decouple at any level $k$. Let $\theta\in Aut(V^k(\mathfrak{n}_4))$ such that it restricts to an involution on the strong generators of $V^k(\mathfrak{sl}_2)$ given by\[\theta(J^+)=J^-, \qquad\theta(J)=-J, \qquad\theta(J^-)=J^+\] and such that $\theta(G^{+})=G^{-}$. These requirements fix the action on the remaining strong generators. The map $\theta$ is an automorphism of $V^k(\mathfrak{n}_4)^{\mathbb{Z}/2\mathbb{Z}}$ \emph{a fortiori} and acting on the decoupling relations of Lemma \ref{LEM: Decoupling relations for the strong generators of the cyclic orbifold that are not in the U(1) orbifold.} shows that the fields $\{\Sigma^{(0)-}_{2n+4,0},\Sigma^{(1)-}_{n+3,0},\Sigma^{(2)-}_{2n+3,0}\}_{n=0}^{\infty}$ decouple at any level $k$ as well. The existence of decoupling relations for all remaining strong generators can be checked directly.\footnote{In our case this was done using Thielemans' Mathematica\texttrademark package \cite{doi:10.1142/S0129183191001001}.} The field $\Sigma^{(1)+}_{2,0}$ decouples at all levels $k\neq4$ (see (\ref{Appendix: Decoupling relation for Sigma^(1+)_(2,0)})). Acting on the decoupling relation of $\Sigma^{(1)+}_{2,0}$ with the automorphism $ \theta$ shows that $\Sigma^{(1)+}_{2,0}$ also decouples at all levels $k\neq4$. Furthermore, the fields $U_{2,0},A_{2,0}$ and $B_{2,0}$ decouple at all levels $k\neq 16$ (see (\ref{Appendix: Decoupling relation for U_(2,0)})-(\ref{Appendix: Decoupling relation for B_(2,0)})) and $V_{2,0}$ decouples at all levels $k\neq0$ (see (\ref{Appendix: Decoupling relation for V^(2,0) Part 1}) and (\ref{Appendix: Decoupling relation for V^(2,0) Part 2})). These exhaust all decoupling relations for the minimal strong generators which proves the Theorem. 
\end{proof}

\section{Structure of the vertex algebras}\label{sec:Structure of the vertex algebra}
We will now look at sub-structures and simple quotients of the two orbifolds of $V^k(\mathfrak{n}_4)$ that were constructed in the previous section. It will be helpful to define the following: Let $R^i\in V^k(\mathfrak{n}_4)$ be a field and define
\begin{equation*}
	C^i=(Q^+)_{(0)} R^i , \qquad D^i=-(Q^-)_{(0)} R^i , \qquad S^i=\frac{1}{2}\left((Q^-) _{(0)} C^i + (Q^+)_{(0)} D^i\right).
\end{equation*}
Note that the Lie superalgebra $\mathfrak{gl}_{1|1}$ has basis $\psi^\pm, N, E$ where $\psi^\pm$ are odd and $N, E$ are even and the non-zero commutation relations are $[N, \psi^\pm]= \pm \psi^\pm$ and $[\psi^+, \psi^-] =E$, so that the algebra spanned by $(Q^\pm)_{(0)}, J_{(0)}$ and $[(Q^+)_{(0)}, (Q^-)_{(0)}]$ is isomorphic to $\mathfrak{gl}_{1|1}$. The definition of $R^i, C^i, D^i$ and $S^i$ thus organizes fields into representations of this $\mathfrak{gl}_{1|1}$-algebra. 

 Taking $R^0=H$ and $R^{i+1}=U_{i,0}$ for $i=0,1,2$ it can be checked that the fields in the set $\mathcal{S}=\{\mathcal{S}^i\}_{i=0}^3$ strongly generate $V^k(\mathfrak{n}_4)^{U(1)}$ at all levels $k\neq0,-2$ (\textit{cf.} Theorem \ref{THM: Minimal strong generators of U(1)-orbifold of V^k(n_4)}). Furthermore, taking $R^3=\Sigma^{(0)+}_{0,0}$ and $R^4=\Sigma^{(1)+}_{1,0}$ as well as $R^{n+2}=\theta(R^n)$ for $n=3,4$ with $\theta$ being the automorphism defined in the proof of Theorem \ref{THM: Main Theorem. minimal strong generators of cyclic orbifold Z/2Z} we see that the set $\{\mathcal{S}^{i}\}_{i=0}^6$ contains all minimal strong generators of $V^k(\mathfrak{n}_4)^{\mathbb{Z}/2\mathbb{Z}}$ at levels $k\neq-2,0,4,16$ (cf. Theorem \ref{THM: Main Theorem. minimal strong generators of cyclic orbifold Z/2Z}).
\begin{thm}
\label{THM: simple quotients of U(1)-orbifold}
	Let $k\neq0,-2$. For all but finitely many levels $k$ the simple quotient $L_k(\mathfrak{n}_4)^{U(1)}$ is of type $\mathcal{W}(1,2^2,3^2,4^2,5; \frac{3}{2}^2 ,\frac{5}{2}^2 ,\frac{7}{2}^2,\frac{9}{2}^2)$. The full list of exceptions is stated in the following table.
	
	\begin{figure}[H]
		\footnotesize
		\makebox[\linewidth]{
		  \begin{tabular}{| r | c | l |}
		    \bottomrule
		    \hline
		    \multicolumn{1}{|c|}{level $k$} & \multicolumn{1}{|c|}{central charge} & \multicolumn{1}{|c|}{type}\\ 
		    \hline  
		    \toprule
		  	$-\frac{5}{2}$ 	& $-15$ & $\mathcal{W}(1,2^2,3^2,4; \frac{3}{2}^2 ,\frac{5}{2}^2 ,\frac{7}{2}^2 )$ \\ 
		  	$-\frac{3}{2}$  & $ -9$ & $\mathcal{W}(1,2,4,5; \frac{3}{2}^2,\frac{9}{2}^2)$ \\ 
		  	$-\frac{4}{3}$ 	& $ -8$ & $\mathcal{W}(1,2^2,3^2,4; \frac{3}{2}^2, \frac{5}{2}^2,\frac{7}{2}^2 )$ \\ 
		  	$-\frac{2}{3}$  & $ -4$ & $\mathcal{W}(1,2^2,3^2,4^2; \frac{3}{2}^2, \frac{5}{2}^2, \frac{7}{2}^2,\frac{9}{2}^2)$ \\ 
		  	$-\frac{1}{2}$  & $ -3$ & $\mathcal{W}(1,2^2,3; \frac{3}{2}^2,\frac{5}{2}^2)$ \\ 
		  	$1$  			& $  6$ & $\mathcal{W}(1,2; \frac{3}{2}^2)$ \\
		  	$2$   			& $ 12$ & $\mathcal{W}(1,2^2,3; \frac{3}{2}^2,\frac{5}{2}^2)$ \\ 
		  	$3$ 			& $ 18$ & $\mathcal{W}(1,2^2,3^2,4; \frac{3}{2}^2,\frac{5}{2}^2,\frac{7}{2}^2)$ \\ 
		  	$4$   			& $ 24$ & $\mathcal{W}(1,2^2,3^2,4^2; \frac{3}{2}^2,\frac{5}{2}^2 ,\frac{7}{2}^2, \frac{9}{2}^2)$ \\[2mm] 
			 \hline
		  \end{tabular}
		}
		\newline
	\end{figure}
\end{thm}
\begin{proof}
	It is straightforward to establish a level dependent basis for the vector space of singular fields at a fixed weight using \cite{doi:10.1142/S0129183191001001}. Let $S\in V^k(\mathfrak{n}_4)^{U(1)}$. A singular field of the form $S + \cdots$ where the ellipsis indicate a sum of normally ordered products induces a decoupling relation for the field $S$ in the simple quotient. The type of the simple quotient can therefore be determined by obtaining all possible levels which contain singular fields of the form $S+\cdots$ for which the field $S$ is a minimal strong generator. All relevant singular fields are listed in Appendix \ref{app: singular fields in U(1)-orbifold}. Note that the $V^k(\mathfrak{n}_2)$-module structure induces decoupling relations for further minimal strong generators.
\end{proof}
\begin{thm}
\label{THM: simple quotients of cyclic orbifold}
	Let $k\neq0,-2$. For all but finitely many levels $k$ the simple quotient $L_k(\mathfrak{n}_4)^{\mathbb{Z}/2\mathbb{Z}}$ is of type $\mathcal{W}(1,2^4,3^2,4^5; \frac{3}{2}^2, \frac{5}{2}^4,\frac{7}{2}^4)$. The full list of exceptions is stated in the following table.
	\begin{figure}[H]
		\footnotesize
		\makebox[\linewidth]{
		  \begin{tabular}{| r | c | l |}
		    \bottomrule
		    \hline
		    \multicolumn{1}{|c|}{level $k$} & \multicolumn{1}{|c|}{central charge} & \multicolumn{1}{|c|}{type}\\ 
		    \hline
		    \toprule
		  	$-\frac{5}{2}$ 		& $-15$ & $\mathcal{W}(1,2^4,3^2,4^4; \frac{3}{2}^2,\frac{5}{2}^4,\frac{7}{2}^4)$ \\ 
		  	$-\frac{3}{2}$  	& $ -9$ & $\mathcal{W}(1,2^3; \frac{3}{2}^2, \frac{5}{2}^2)$ \\ 
		  	$-\frac{4}{3}$ 		& $ -8$ & $\mathcal{W}(1,2^4,3^2,4^4; \frac{3}{2}^2,\frac{5}{2}^4,\frac{7}{2}^4)$ \\ 
		  	$-\frac{1}{2}$  	& $ -3$ & $\mathcal{W}(1,2^4,3; \frac{3}{2}^2,\frac{5}{2}^4)$ \\ 
		  	$1$  				& $  6$ & $\mathcal{W}(1,2,4^4; \frac{3}{2}^2,\frac{7}{2}^2)$ \\
		  	$2$   				& $ 12$ & $\mathcal{W}(1,2^4,3,4^2; \frac{3}{2}^2, \frac{5}{2}^4,\frac{7}{2}^2)$ \\ 
		  	$3$ 				& $ 18$ & $\mathcal{W}(1,2^4,3^2,4^3; \frac{3}{2}^2,\frac{5}{2}^4,\frac{7}{2}^4)$ \\ 
		  	$4$   				& $ 24$ & $\mathcal{W}(1,2^4,3^2,4^5; \frac{3}{2}^2, \frac{5}{2}^4, \frac{7}{2}^4, \frac{9}{2}^2)$ \\ 
		  	$16$   				& $ 96$ & $\mathcal{W}(1,2^4,3^2,4^6; \frac{3}{2}^2, \frac{5}{2}^4, \frac{7}{2}^4, \frac{9}{2}^2)$\\[2mm]
			 \hline
		  \end{tabular}
		}
		\newline
	\end{figure}
\end{thm}
\begin{proof}
	The proof is analogous to the proof of Theorem \ref{THM: simple quotients of U(1)-orbifold}. All relevant singular fields are listed in Appendix \ref{app: singular fields in cyclic orbifold}. Again, note that the $V^k(\mathfrak{n}_2)$-module structure as well as the action of the automorphism $\theta$ as defined in the proof of Theorem \ref{THM: Main Theorem. minimal strong generators of cyclic orbifold Z/2Z} induce decoupling relations for further minimal strong generators. 
\end{proof}

Given that the proofs of Theorems \ref{THM: simple quotients of U(1)-orbifold} and \ref{THM: simple quotients of cyclic orbifold} are purely computational some remarks are in order:

\begin{rem}\label{rem:integer}
	 It is apparent from the singular fields in appendix \ref{app: singular fields in cyclic orbifold} that for $k=1$, $L_k(\mathfrak{n}_4)$ admits an action of the simple vertex algebra $L_k(\mathfrak{sl}_2)$ at level $k=1$. This statement can also be seen using free field realizations of $L_1(\mathfrak n_4)$, see \cite[Lemma 3.4]{Creutzig:2013mqa}. For positive integer $n$, $L_{n+1}(\mathfrak{sl}_2)$ embeds in $L_n(\mathfrak{sl}_2) \otimes L_1(\mathfrak{sl}_2)$ and since $\mathfrak n_4$ is a Lie superalgebra also a homomorphic image of $V^{n+1}(\mathfrak n_4)$ embeds into $L_n(\mathfrak n_4) \otimes L_1(\mathfrak n_4)$. It thus follows that this homomorphic image of $V^{n+1}(\mathfrak n_4)$ containts a copy of $L_{n+1}(\mathfrak{sl}_2)$ and so especially the simple quotient $L_{n+1}(\mathfrak n_4)$ containts a copy of $L_{n+1}(\mathfrak{sl}_2)$
\end{rem}

\begin{rem}
The levels $-\frac{1}{2}(3+2n)$ for a positive integer $n$, are also special. We will see in Theorem \ref{thm:Wiso} that at these levels, an orbifold of a coset of $L_k(\mathfrak{n}_4)$ is a principal $\mathcal{W}$-algebra of type $A$. 
The special cases $k=-\frac{1}{2}$ and $k=-\frac{3}{2}$ are already well understood. Namely, $L_{-\frac{1}{2}}(\mathfrak{n}_4) \cong (\mathcal A(1) \otimes \mathcal S(1))^{\mathbb Z/2\mathbb Z}$ by \cite[Thm 4.14]{CKL}. Here $\mathcal A(1)$ is the rank one symplectic fermion algebra and $\mathcal S(1)$ the rank one $\beta\gamma$ system.	
The construction of $L_k(\mathfrak{n}_4)$ at level $k=-\frac{3}{2}$ is first given in \cite{2014arXiv1407.1527A}, and in Theorem 2.5 of \cite{CGL} it is shown that
\[
L_{-\frac{3}{2}}(\mathfrak{n}_4) \cong \bigoplus_{n=0}^\infty L_{-\frac{3}{2}}(n\omega) \otimes \rho_{n\omega}
\]
as $L_{-\frac{3}{2}}(\mathfrak{sl}_2) \otimes SU(2)$-modules. In this notation, $\omega$ denotes the fundamental weight of $\mathfrak{sl}_2$, $\rho_{n\omega}$ denotes the irreducible $\mathfrak{sl}_2$-module with highest weight $n\omega$, and $L_{-\frac{3}{2}}(n\omega)$ denotes the corresponding irreducible $L_{-\frac{3}{2}}(\mathfrak{sl}_2)$-module.

Let us also note that this series of special points is suggested in \cite{Bonetti:2018fqz} to be subalgebras of the chiral algebras of certain four-dimensional super Yang-Mills theories.
\end{rem}

\begin{cor}
\label{cor:enriques}
The vertex algebra of global sections of the chiral de Rham complex on a complex Enriques surface is of type $\mathcal{W}(1,2,4^4; \frac{3}{2}^2,\frac{7}{2}^2)$. It is strongly generated by the fields \[J(z),Q^{\pm}(z),T(z),\Sigma^{(1)\pm}_{1,0},\Sigma^{(0)\pm}_{2,0},\Sigma^{(2)\pm}_{1,0}.\]
\end{cor}

\begin{proof} This follows immediately from Theorem 6.6 of \cite{2018arXiv180910320S}, together with the description of $L_1(\mathfrak{n}_4)^{\mathbb{Z}_2}$ given in Theorem \ref{THM: simple quotients of cyclic orbifold}.
\end{proof}

\section{Coset of $V^k(\mathfrak{n}_4)$ by its affine subalgebra}\label{sec:7}
In this section, we study the coset
\begin{equation} \mathcal{C}^k = \text{Com}(V^k(\mathfrak{sl}_2),  V^k(\mathfrak{n}_4)), \end{equation} and we regard $V^k(\mathfrak{n}_4)$ as an extension of $V^k(\mathfrak{sl}_2) \otimes \mathcal{C}^k$. In Theorem 5.4 of \cite{ACKL17}, $\mathcal{C}^k$ was incorrectly stated to be of type $\mathcal{W}(2,3^3,4,5^3,6,7^3,8)$. In this section, we give the correct description as well as some more details about its structure.

As in Section 4 of \cite{ACKL17}, if we rescale the generators of $ V^k(\mathfrak{n}_4)$ by $\frac{1}{\sqrt{k}}$, there is a well-defined limit as $k \rightarrow \infty$, and
$$\lim_{k\rightarrow \infty} V^k(\mathfrak{n}_4) \cong \mathcal{H}(3) \otimes \mathcal{T} \otimes \mathcal{G}_{\text{odd}}(4).$$

In this notation, $\mathcal{H}(3)$ is the rank $3$ Heisenberg vertex algebra, $ \mathcal{T}$ is a generalized free field algebra with one even generator $T$ satisfying $T(z) T(w) \sim \frac{6}{(z-w)^{4}}$, and $ \mathcal{G}_{\text{odd}}(4)$ is the generalized free field algebra with odd generators $G^{\pm}, Q^{\pm}$ satisfying
$$G^+(z) G^{-}(w) \sim \frac{2}{(z-w)^{3}},\qquad Q^+(z) Q^{-}(w) \sim \frac{2}{(z-w)^{3}}.$$ Note that our normalizations of the generator differ slightly from those in \cite{ACKL17}, but this does not change the above result. Note that the action of the inner automorphism group $G_{\text{Inn}} \cong SL_2$ on $V^k(\mathfrak{n}_4)$ coming from integrating the zero-mode action of $\mathfrak{sl}_2$ gives rise to action of $SL_2$ on $\mathcal{G}_{\text{odd}}(4)$, such that $\{G^+, Q^+\}$ and $\{G^-, Q^-\}$ both transform as copies of the standard module $\mathbb{C}^2$.

As shown in \cite{ACKL17} right before Theorem 4.12, $\mathcal{C}^k$ has a well-defined limit as $k \rightarrow \infty$, and
 $$\lim_{k\rightarrow \infty} \mathcal{C}^k \cong \mathcal{T} \otimes \big( \mathcal{G}_{\text{odd}}(4) \big)^{SL_2}.$$
Moreover, structure of $\big( \mathcal{G}_{\text{odd}}(4) \big)^{SL_2}$ can be worked out using classical invariant theory. First, we have the infinite generating set
\begin{equation} \begin{split} m^j & =\  :G^+ \partial^j Q^+: \ +\  :\partial^j G^+ Q^+:,
\\ p^j & =\  :G^- \partial^j Q^-: \ +\  :\partial^j G^- Q^-:,
\\ w^j & = \ :G^+ \partial^j G^-: \ +\  :Q^+ \partial^j Q^-:. \end{split} \end{equation} 
for $j\geq 0$. Note that $m^j, p^j, w^j$ each have weight $j+3$. It is straightforward to check that $\{w^0, m^0, p^0\}$ generates the algebra $( \mathcal{G}_{\text{odd}}(4))^{SL_2}$, and that the set $\{w^i, m^j, p^j|\ i = 0,1,2,3,\ j = 0,2\}$ close under OPE, and hence strongly generates the algebra. We obtain

\begin{thm} $\big( \mathcal{G}_{\text{odd}}(4) \big)^{SL_2}$ is of type $\mathcal{W}(3^3,4,5^3,6)$, so that $\mathcal{C}^k$ is of type $\mathcal{W}(2,3^3,4,5^3,6)$ for generic values of $k$.
\end{thm}

In \cite{ACKL17}, it was also stated (correctly) that $\mathcal{C}^k$ contains a vertex subalgebra of type $\mathcal{W}(2,3,4,5,6,7,8)$. We now give more details about this subalgebra. First of all, inside the outer automorphism group $G_{\text{Out}} \cong SL_2$, there is a copy of $U(1)$, and a corresponding outer action of the one-dimensional abelian Lie algebra $\mathfrak{t}$. Note that the fields $w^i, m^j, p^j$ are eigenvectors with eigenvalue $0, -2, 2$ under this action, respectively. It follows that the orbifold $(\mathcal{C}^k)^{U(1)}$ is strongly generated by the fields $\{w^i|\ i\geq 0\}$ together with all monomials 
$$:\partial^{a_1} p^{j_1} \partial^{a_2} p^{j_2} \cdots \partial^{a_s} p^{j_s} \partial^{b_1} m^{k_1} \partial^{b_2} m^{k_2} \cdots \partial^{b_s} m^{k_s}:,$$ where $a_i, b_i, j_i, k_i$ are nonnegative integers, and $s\geq 1$.

Moreover, one can verify by computer calculation that the fields $\{w^i|\ i = 0,1,2,3,4,5\}$ close under OPE, and that for $a,b\geq 0$ and $j,k = 0,2$, the field $:\partial^a p^j \partial^b m^k:$ lies in the subalgebra generated by $\{w^i|\ i = 0,1,2,3,4,5\}$. From this observation, and by induction on $s$, we obtain

\begin{thm} $(\mathcal{C}^k)^{U(1)}$ is strongly generated by fields $\{w^i| \ i = 0,1,2,3,4,5\}$, and hence is of type $\mathcal{W}(2,3,4,5,6,7,8)$.
\end{thm} 

We may take the weight $3$ field $w^0$ to be primary of weight $3$ and we normalize it so that its sixth order pole with itself is $\displaystyle \frac{k (3 + 2 k)}{2 + k} = \frac{c}{3}$. Following the notation in \cite{L17}, we denote this field by $W^3$; it has the explicit form 
$$W^3 = \frac{1}{\sqrt{8 + 4 k}} (:G^+ G^-: + :Q^+ Q^-: - \partial T).$$ Moreover, it is not difficult to verify that it generates $(\mathcal{C}^k)^{U(1)}$. Following the convention of \cite{L17}, we may take the strong generating set for $(\mathcal{C}^k)^{U(1)}$ to be $\{L, W^i|\ i = 3,4,5,6,7,8\}$, where $W^i = W^3_{(1)} W^{i-1}$ for $i = 4,5,6,7,8$.

It is readily verified that the hypotheses of Theorem 6.2 of \cite{L17} are satisfied, so that $(\mathcal{C}^k)^{U(1)}$ can be realized as a quotient of $\mathcal{W}(c, \lambda)$ of the form $\mathcal{W}_I(c,\lambda) = \mathcal{W}^I(c,\lambda) / \mathcal{I}$. In this notation, $I\subseteq \mathbb{C}[c,\lambda]$ is some prime ideal in the ring of parameters $\mathbb{C}[c,\lambda]$, and $\mathcal{I}$ is the maximal proper graded ideal of $\mathcal{W}^I(c,\lambda) = \mathcal{W}(c,\lambda) / I \cdot \mathcal{W}^I(c,\lambda)$.

By computing the third order pole of $W^3$ with itself, it is straightforward to verify that $I$ is the ideal $(\lambda + \frac{1}{16})$. Rather surprisingly, this same vertex algebra was studied in Section 11 of \cite{L17}. Combining this calculation with Corollary 10.3 of \cite{L17}, we obtain

\begin{thm} \label{ckorbifoldcoinc} $(\mathcal{C}^k)^{U(1)}$ is isomorphic to the coset $$\text{Com}(V^{\ell}(\mathfrak{sl}_2), V^{\ell+1}(\mathfrak{sl}_2) \otimes \mathcal{W}_{-5/2}(\mathfrak{sl}_4, f_{\text{rect}})),$$ where the parameters $k$ and $\ell$ are related by $\displaystyle k = -\frac{\ell+1}{\ell+2}$. \end{thm} 
\begin{rem}
This Theorem nicely relates to the coset realization of $V^k(\mathfrak n_4)$ of \cite{CFL}.
Let $L_{k}(n\omega)$ denote the irreducible highest-weight module of $\widehat{\mathfrak{sl}}_2$ of highest weight $n\omega$ at level $k$. $\omega$ is the fundamental weight of $\mathfrak{sl}_2$ and $\rho_{m\omega}$ denotes the irreducible highest-weight module of $SU(2)$ of highest weight $m\omega$. Also let $\bar n$ be equal to $0$ if $n$ is even and $1$ otherwise.
 We have the following list of isomorphisms
\begin{enumerate}
\item In \cite[Section 5]{CKLR} diagonal Heisenberg cosets of the tensor product of a rank $n$ $\beta\gamma$ system with a rank $m$ $bc$ system where studied. These cosets were denoted by $C(n, m)$ and $C(2, 0)\cong \mathcal{W}_{-5/2}(\mathfrak{sl}_4, f_{\text{rect}})$ \cite[Remark 5.3]{CKLR} and $C(2, 2) \cong L_1(\mathfrak{sl}_{2|2})$ \cite[Theorem 5.5]{CKLR}. Moreover $C(0, 2)$ is nothing but the lattice VOA $L_1(\mathfrak{sl}_2)$ and so we have that 
$\mathcal{W}_{-5/2}(\mathfrak{sl}_4, f_{\text{rect}})  \subset \text{Com}\left(L_1(\mathfrak{sl}_2), L_1(\mathfrak{sl}_{2|2})\right)$ and by passing to the simple quotient $L_1(\mathfrak{psl}_{2|2})$ of $L_1(\mathfrak{sl}_{2|2})$ we also have $\mathcal{W}_{-5/2}(\mathfrak{sl}_4, f_{\text{rect}})  \subset \text{Com}\left(L_1(\mathfrak{sl}_2), L_1(\mathfrak{psl}_{2|2})\right)$. The branching rules \cite[Cor. 5.3]{Creutzig2017} and \cite[Rem. 9.11]{CGai} 
\begin{equation}\label{eq:decomp}
\mathcal{W}_{-5/2}(\mathfrak{sl}_4, f_{\text{rect}})  \cong \bigoplus_{n=0}^\infty L_{-1}(2m\omega)
\end{equation}
and
\[
L_1(\mathfrak{psl}_{2|2}) \cong \bigoplus_{n=0}^\infty L_{-1}(n\omega) \otimes \rho_{n\omega} \otimes L_1(\bar n\omega)
\]
 tell us that 
\[
\mathcal{W}_{-5/2}(\mathfrak{sl}_4, f_{\text{rect}})  \cong \text{Com}\left(L_1(\mathfrak{sl}_2), L_1(\mathfrak{psl}_{2|2})\right)^{U(1)}.  
\]
\item In \cite{CFL} a vertex superalgebra $Y(\lambda)$ that is related to $L_1(\mathfrak{d}(2, 1; -\lambda)) \otimes L_1(\mathfrak{psl}_{2|2})$ has been constructed. It satisfies
\[
Y(\lambda)^{\mathbb Z/2\mathbb Z} \cong \text{Com}\left(L_1(\mathfrak{sl}_2) \otimes L_1(\mathfrak{sl}_2)  , L_1(\mathfrak{d}(2, 1; -\lambda)) \otimes L_1(\mathfrak{psl}_{2|2})\right).
\]
\item 
$Y(\lambda)$ decomposes
\[
Y(\lambda) := \bigoplus_{ n, m=0}^\infty L_{k_1}(n\omega) \otimes L_{k_2}(n\omega) \otimes L_{-1}(m\omega) \otimes \rho_{m\omega} 
\]
as $L_{k_1}(\mathfrak{sl}_2) \otimes L_{k_2}(\mathfrak{sl}_2) \otimes L_{-1}(\mathfrak{sl}_2) \otimes SU(2)$-module for generic complex $\lambda$. Here $k_1=\lambda^{-1}-1, k_2=\lambda-1$ and note that for generic $\lambda$ the universal affine vertex operator algebras of $\mathfrak{sl}_2$  at levels $k_1$ and $k_2$ are simple. Then by \cite{CFL} and for irrational $\lambda$ we have that 
\[
V^{k_2}(\mathfrak n_4) \cong  \text{Com}\left(L_{k_1 -1}(\mathfrak{sl}_2), Y(\lambda)\right).
\]
\item Putting all these together we get 
\begin{equation}\nonumber
\begin{split}
(\mathcal{C}^{k_2})^{U(1)} &\cong  \text{Com}\left(L_{k_2}(\mathfrak{sl}_2) \otimes L_{k_1 -1}(\mathfrak{sl}_2), Y(\lambda)\right)^{U(1)} \\
&\cong  \text{Com}\left( L_{k_2}(\mathfrak{sl}_2) \otimes L_{k_1 -1}(\mathfrak{sl}_2), \bigoplus_{ n, m=0}^\infty L_{k_1}(n\omega) \otimes L_{k_2}(n\omega) \otimes L_{-1}(m\omega) \otimes \rho_{m\omega})\right)^{U(1)} \\
&\cong  \text{Com}\left( L_{k_1 -1}(\mathfrak{sl}_2), \bigoplus_{ m=0}^\infty L_{k_1}(\mathfrak{sl}_2)  \otimes L_{-1}(m\omega) \otimes \rho_{m\omega})\right)^{U(1)} \\
&\cong  \text{Com}\left( L_{k_1 -1}(\mathfrak{sl}_2), \bigoplus_{ m=0}^\infty L_{k_1}(\mathfrak{sl}_2)  \otimes L^{-1}(2m\omega))\right).
\end{split}
\end{equation}
which using \eqref{eq:decomp} and noticing that $k_2 = -\frac{k_1}{k_1+1}$ nicely compares to the Theorem.
\end{enumerate}
\end{rem}

We now present some consequences of the identification of $(\mathcal{C}^k)^{U(1)}$ with a quotient of $\mathcal{W}(c,\lambda)$. Recall from Section 10 of \cite{L17}, we can obtain coincidences between the simple quotient of $(\mathcal{C}^k)^{U(1)}$ with various other algebras arising as quotients of $\mathcal{W}(c,\lambda)$ by finding the intersection points on their truncation curves.

Recall that if we regard $\mathcal{C}^k$ as a one-parameter vertex algebra, with $k$ a formal variable, the specialization of $\mathcal{C}^k$ at a complex number $k = k_0$ need not coincide with the actual coset, but this can only fail when $k_0 + 2 \in \mathbb{Q}_{\leq 0}$. This property is inherited by the orbifold $(\mathcal{C}^k)^{U(1)}$ if we also omit the point $k_0 = 0$. By abuse of notation, in the results below, $(\mathcal{C}^{k_0})^{U(1)}$ will always refer to the specialization of the one-parameter vertex algebra $(\mathcal{C}^k)^{U(1)}$ at the point $k = k_0$, even if is strictly larger than the actual algebra $\text{Com}(V^{k_0}(\mathfrak{sl}_2),  V^{k_0}(\mathfrak{n}_4))^{U(1)}$. We also denote by $(\mathcal{C}_{k_0})^{U(1)}$ the simple quotient of $(\mathcal{C}^{k_0})^{U(1)}$.

The next result follows immediately from Theorem \ref{ckorbifoldcoinc} and Theorem 11.4 of \cite{L17}

\begin{thm}\label{thm:Wiso} For $n\geq 3$, aside from the critical levels $k = -2$ and $\ell = -n$, and the degenerate cases given by Theorem 10.1 of \cite{L17}, all isomorphism 
$(\mathcal{C}^k)^{U(1)}\cong \mathcal{W}_{\ell}(\mathfrak{sl}_n, f_{\text{prin}})$ appear on the following list.
\begin{equation} k = -\frac{1}{2} (n+2),\qquad k = -\frac{2 (n-1)}{n-2}, \qquad \ell =  -n + \frac{n-2}{n} ,\qquad \ell = -n + \frac{n}{n-2},\end{equation} which has central charge 
$\displaystyle c =-\frac{3 (n-1) (n+2)}{n-2}$.
\end{thm}

Next, in the terminology of \cite{L17}, recall the {\it generalized parafermion algebra} 
$$\mathcal{G}^{\ell}(n) = \text{Com}(V^{\ell}(\mathfrak{gl}_n), V^{\ell}(\mathfrak{sl}_{n+1})),$$ and its simple quotient $\mathcal{G}_{\ell}(n)$. By Theorem 8.2 of \cite{L17}, this also arises a quotient of $\mathcal{W}(c,\lambda)$ and the corresponding truncation curve is given explicitly by (8.4) of \cite{L17}. Additionally, by (8.5) of \cite{L17}, this curve has the following rational parametrization using the level $\ell$ as parameter:
\begin{equation} \label{param:gp} \lambda(\ell)  = \frac{(n + \ell) (1 + n + \ell)}{(\ell-2 ) (2 n + \ell) (2 + 2 n + 3 \ell)},\qquad c(\ell) = \frac{n (\ell-1) (1 + n + 2 \ell)}{(n + \ell) (1 + n + \ell)} .\end{equation} 

\begin{thm} \label{genparacoincid} For $n \geq 3$, aside from the critical levels $k = -2$, $\ell = -n$, and $\ell = -n-1$, and the degenerate cases given by Theorem 10.1 of \cite{L17}, all isomorphisms $(\mathcal{C}_k)^{U(1)} \cong \mathcal{G}_{\ell}(n)$ appear on the following list.

\begin{enumerate}

\item  $\displaystyle k= n,\ k = - \frac{3 + 2 n}{n+2}, \qquad \ell =  - 2 (1 + n)$, \\ which has central charge 
$\displaystyle c = \frac{3 n (3 + 2 n)}{2 + n}$. 

\item $\displaystyle k = -\frac{n-3}{n-2},\ k = -\frac{n}{n-1},\qquad \ell =  - 2$,\\ which has central charge 
$\displaystyle c = -\frac{3n (n-3)}{(n-2) (n-1)}$.

\item  $\displaystyle k= \frac{1}{3} (n-3),\ k = - \frac{3 + 2 n}{3 + n},\qquad  \ell =  -\frac{2 n}{3}$,\\ which has central charge 
$\displaystyle c = \frac{(n-3) (3 + 2 n)}{3 + n}$.

\end{enumerate}
\end{thm}

\begin{proof} We first exclude the values $\ell =  2, \ -2n,\  -\frac{1}{3}(2m+2)$ which are poles of function $\lambda(\ell)$ given by \eqref{param:gp}. As explained in \cite{L17}, at these points, $\mathcal{G}^{\ell}(n)$ is not obtained as a quotient of $\mathcal{W}(c,\lambda)$ at these points. Note that the truncation curve for $(\mathcal{C}_k)^{U(1)}$ has parametrization $$\displaystyle c(k) = \frac{3 k (3 + 2 k)}{2 + k},\qquad \lambda = -\frac{1}{16},$$ and since the pole $k=-2$ has already been excluded, there are no additional points where $(\mathcal{C}_k)^{U(1)}$ cannot be obtained as a quotient of $\mathcal{W}(c,\lambda)$. By Corollary 10.2 of \cite{L17}, aside from the cases $c = 0,-2$, all remaining isomorphisms $(\mathcal{C}_k)^{U(1)} \cong \mathcal{G}_{\ell}(n)$ correspond to intersection points on the curves $V(K_m)$ and $V(I)$, where $K_m$ is given by (8.4) of \cite{L17}, and $I= (\lambda + \frac{1}{16})$, as above. For each $n \geq 2$, there are exactly three intersection points $(c, \lambda)$, namely,
$$  \bigg(\frac{3 n (3 + 2 n)}{2 + n}, -\frac{1}{16}\bigg),\qquad \bigg(-\frac{3n (n-3)}{(n-2) (n-1)}, -\frac{1}{16}\bigg),\qquad \bigg(\frac{(n-3) (3 + 2 n)}{3 + n}, -\frac{1}{16}\bigg).$$
It is immediate that the above isomorphisms all hold, and that our list is complete except for possible coincidences at the excluded points $\ell =  2, \ -2n,\  -\frac{1}{3}(2n+2)$.
 
At $\ell = 2$, $\mathcal{G}_{\ell}(n)$ has central charge $\displaystyle c = \frac{n (5 + n)}{(2 + n) (3 + n)}$ and the weight $3$ field is singular. However, the weight $3$ field in $(\mathcal{C}_k)^{U(1)}$ is not singular at this central charge, so there is no coincidence at this point. Similarly, at $\ell = - 2n$ and $\ell =  -\frac{1}{3}(2n+2)$, $\mathcal{G}_{\ell}(n)$ has central charge $\displaystyle c = \frac{(2n+1) (3n-1)}{n-1}$ and $\displaystyle c = \frac{n (2n+5)}{n-2}$, respectively, and has a singular vector in weight $3$, but at these central charges, $(\mathcal{C}_k)^{U(1)}$ does not. Therefore there are no additional coincidences at these points. \end{proof}

\begin{rem} \label{enriquesparafermions} The first family (1) in Theorem \ref{genparacoincid} is of particular interest since it concerns the case where $k$ is a positive integer $n$. By Remark \ref{rem:integer}, the map $V^n(\mathfrak{sl}_2) \rightarrow V^n(\mathfrak{n}_4)$ descends to a map of simple vertex algebras $L_{n}(\mathfrak{sl}_2) \rightarrow L_n(\mathfrak{n}_4)$. By Corollary 2.2 of \cite{ACKL17}, the coset $\text{Com}(L_{n}(\mathfrak{sl}_2) ,  L_n(\mathfrak{n}_4))$ is simple, and hence coincides with the simple quotient $\mathcal{C}_n$ of $\mathcal{C}^n$. Moreover, by \cite{doi:10.1155/S1073792896000566}, the simplicity of $\mathcal{C}_n$ is preserved by taking the $U(1)$-orbifold. It follows that for all $n\in \mathbb{N}$,
$$\text{Com}(L_{n}(\mathfrak{sl}_2) ,  L_n(\mathfrak{n}_4)) \cong \mathcal{G}_{ - 2 (1 + n)}(n).$$

In the case $n = 1$, note that $\mathcal{G}_{-4}(1)$ is just the parafermion algebra $N_{-4}(\mathfrak{sl}_2) = \text{Com}(\mathcal{H}, L_{-4}(\mathfrak{sl}_2))$. Therefore $L_1(\mathfrak{n}_4)$ may be regarded as an extension of $L_1(\mathfrak{sl}_2) \otimes N_{-4}(\mathfrak{sl}_2)$. Likewise, $L_1(\mathfrak{n}_4)^{\mathbb{Z}/2\mathbb{Z}}$, which is isomorphic to the global section algebra of the chiral de Rham complex of an Enriques surface, is an extension of $\mathcal{H} \otimes N_{-4}(\mathfrak{sl}_2)$, where $\mathcal{H}$ is the Heisenberg algebra generated by $J$.

Consider the coset $\text{Com}\left(V^{k+2}(\mathfrak{sl}_n) \otimes \mathcal H(1),  L_k(\mathfrak{sl}_{n+1}) \otimes \mathcal E(2n)\right)$ where $\mathcal E(2n)$ denotes the $bc$-system of rank $2n$ and  the Heisenberg algebra action is taken to be the diagonal one in such a way that this coset has four odd dimension $3/2$ fields. Its weight one subspace is $\mathcal H(1) \otimes L_n(\mathfrak{sl}_2)$ and if we specialize to $k=-2(n+1)$ then it is easy to check that the $\mathcal H(1)$ becomes central and so by uniqueness of minimal $\mathcal{W}$-superalgebras \cite[Thm. 3.1]{ACKL17} at this level the coset contains $L_n(\mathfrak{n}_4)$ as subalgebra. This fits into the observation of the first family (1) in Theorem \ref{genparacoincid} as this coset also obviously contains $\mathcal{G}_{ - 2 (1 + n)}(n)$.

This observation somehow extends to negative levels and thus connects Theorems \ref{thm:Wiso} and \ref{genparacoincid}. For this consider the coset $\text{Com}\left(V^{k-2}(\mathfrak{sl}_n) \otimes \mathcal H(1),  L_k(\mathfrak{sl}_{n|1}) \otimes \mathcal S(2n)\right)$. The rank $2n$ $\beta\gamma$-system $ \mathcal S(2n)$ carries an action of $V^{-n}(\mathfrak{sl}_2) \otimes V^{-2}(\mathfrak{sl}_n) \otimes \mathcal H(1)$ and in the commutant we choose the Heisenberg diagonally so that the coset has four dimension $3/2$ fields. As in the previous case the weight one subspace is $\mathcal H(1) \otimes V^{-n}(\mathfrak{sl}_2)$ and if we specialize to $k=-2(n-1)$ then it is easy to check that the $\mathcal H(1)$ becomes central and so by uniqueness of minimal $\mathcal{W}$-superalgebras \cite[Thm. 3.1]{ACKL17} at this level the coset contains a homomorphic image of $V^{-n}(\mathfrak{n}_4)$ as subalgebra. This coset also contains $\mathcal {SG}^k(n):=\text{Com}\left(V^{k}(\mathfrak{sl}_n) \otimes \mathcal H(1),  L_k(\mathfrak{sl}_{n|1}) \right)$ as subalgebra, and its central charge for $k=-2(n-1)$ is precisely $\displaystyle -\frac{3(2n^2-3n)}{n-2}$, which is the central charge of
$(\mathcal{C}_{-n})^{U(1)}\cong\mathcal{W}_{\ell}(\mathfrak{sl}_{2(n-1)}, f_{\text{prin}})$  at $\displaystyle \ell= -2(n-1) +\frac{n-2}{n-1}$. 
This observation actually leads us to a new level-rank duality that we will introduce in section \ref{sec:levelrank}.
\end{rem}

There is another interesting family of vertex algebras that arise as quotients of $\mathcal{W}(c,\lambda)$, namely, the cosets
$$\mathcal{D}^{\ell}(n) = \text{Com}(\mathcal{H}, \mathcal{W}^{\ell}(\mathfrak{sl}_n, f_{\text{subreg}});$$ see 
Corollary 6.15 of \cite{CL20}. The explicit truncation curve can be obtained by specializing Theorem 6.1 of \cite{CL20} to the case $m=1$, also appears as Conjecture 9.5 of \cite{L17}. We obtain

\begin{thm} For $n \geq 3$, aside from the critical levels $k = -2$ and $\ell = -n$, and the degenerate cases given by Theorem 10.1 of \cite{L17}, all isomorphisms $(\mathcal{C}_k)^{U(1)} \cong \mathcal{D}_{\ell}(n)$ appear on the following list.

\begin{enumerate}

\item $\displaystyle k = -\frac{n}{1 + n},\ k = - \frac{3 + n}{2 + n},  \qquad \ell = -n + \frac{2 + n}{1 + n}$, which has central charge 
$\displaystyle c =-\frac{3 n (3 + n)}{(1 + n) (2 + n)}$.

\item $\displaystyle k = -n,\ k =  \frac{3 - 2 n}{-2 + n}, \qquad \ell =  -n + \frac{n-2}{n-1}$, which has central charge 
$\displaystyle c = -\frac{3 n (2n-3)}{n-2}$.

\item $\displaystyle k = - \frac{1}{3} (3+ n),\ k =  \frac{3 - 2 n}{n-3},\qquad \ell =  -n + \frac{n}{n-3}$, which has central charge 
$\displaystyle c =-\frac{(3 + n) (2n-3)}{n-3}$.

\end{enumerate}
\end{thm}

The proof is similar to the proof of Theorem \ref{genparacoincid}, and is left to the reader.

\section{Level-rank dualities}\label{sec:levelrank}

We will now explain that the  central charge agreement observed in Remark \ref{enriquesparafermions}  is not a coincidence and it fits into the following bigger picture. 
Firstly, the central charge of the cosets $\mathcal {SG}^{-m}(n)$ and of $\text{Com}\left(V^{-n+1}(\mathfrak{sl}_m),  L_{-n}(\mathfrak{sl}_{m}) \otimes L_1(\mathfrak{sl}_m) \right)$ are both equal to
\[
c = m-1 -\frac{m(m^2-1)}{(m-n)(m-n+1)}.
\]
On the other hand recall that the simple quotient of the coset $$\text{Com}\left(V^{-n+1}(\mathfrak{sl}_m),  L_{-n}(\mathfrak{sl}_{m}) \otimes L_1(\mathfrak{sl}_m) \right)$$
is isomorphic to $\mathcal{W}_{\ell}(\mathfrak{sl}_{m}, f_{\text{prin}})$ at level $\ell = -m +\frac{m-n}{m-n+1}$ by the main Theorem of \cite{ACL19}.

We can be more general, namely consider now 
\[
\mathcal {SG}^{-m}(n|r):=\text{Com}\left(V^{-m}(\mathfrak{sl}_n) \otimes V^{m}(\mathfrak{sl}_r)\otimes  \mathcal H(1),  L_{-m}(\mathfrak{sl}_{n|r}) \right)
\]
together with
\[
\text{Com}\left(V^{-n+r}(\mathfrak{sl}_m),  L_{-n}(\mathfrak{sl}_{m}) \otimes L_r(\mathfrak{sl}_m) \right)
\]
and again there central charges turn out to coincide, i.e. they are equal to 
\[
c = \frac{(m^2-1)nr(n-r-2m)}{(m-n)(m+r)(m+r-n)}.
\]

This observation can be lifted to a new type of level-rank duality. 
For this consider $\mathcal E(mn) \otimes \mathcal S(\ell n)$ and recall that we denote by $\mathcal E(m)$ the $bc$-system of rank $m$ and by $\mathcal S(m)$ the $\beta\gamma$-system of rank $m$. The vertex superalgebra $\mathcal E(mn) \otimes \mathcal S(\ell n)$ is viewed as the $bc\beta\gamma$-system for $\mathbb C^{n} \otimes \mathbb C^{m|\ell}$, i.e. for the tensor product of the standard representations of $\mathfrak{gl}_n$ and $\mathfrak{sl}_{m|\ell}$. It thus carries a commuting action of $V^{m-\ell}(\mathfrak{gl}_n)\cong V^{m-\ell}(\mathfrak{sl}_n) \otimes \mathcal H(1)$ and $V^n(\mathfrak{sl}_{m|\ell})$.  We normalize the Heisenberg field to have norm one, so that the $b, c, \beta$ and $\gamma$ all have Heisenberg weight $\mu=\frac{1}{\sqrt{n(m-\ell)}}$.
The conformal weight $\Delta$ of the module $V^{m-\ell}(\omega_1) \otimes \pi_\mu \otimes V^n(\omega_1)$ of  $V^{m-\ell}(\mathfrak{sl}_n) \otimes \mathcal H(1) \otimes V^n(\mathfrak{sl}_{m|\ell})$ is
\[
\Delta = \frac{(n^2-1)}{2(n+m-\ell)} + \frac{1}{2n(m-\ell)} +  \frac{((m-\ell)^2-1)}{2(n+m-\ell)} = \frac{1}{2}
\]
so that \cite[Corollary 2.2]{AKFPP} applies, i.e. there is a conformal embedding of $V^{m-\ell}(\mathfrak{sl}_n) \otimes \mathcal H(1) \otimes V^n(\mathfrak{sl}_{m|\ell})$ in $\mathcal E(mn) \otimes \mathcal S(\ell n)$. We set
\[
A^n(\mathfrak{sl}_{m|\ell}) := \text{Com}\left(V^{m-\ell}(\mathfrak{gl}_n), \mathcal E(mn) \otimes \mathcal S(\ell n) \right)
\]
and if $m=0$, then we write $A^{-n}(\mathfrak{sl}_{\ell})$ for $A^n(\mathfrak{sl}_{0|\ell})$.
For descriptions of some of the cosets of these types see \cite{LSS}. We also need
 \cite[Thm. 4.1]{OstrikSun}, i.e. 
\begin{equation}
\begin{split}
\text{Com}\left(L_n(\mathfrak{gl}_m), \mathcal E(mn) \right) &\cong L_m(\mathfrak{sl}_n).
\end{split}
\end{equation}
With this notation and information we can slightly modify  the argument of \cite[Thm. 13.1]{ACL19} to get
\begin{equation}\nonumber
\begin{split}
&\text{Com}\left(V^{-n+r}(\mathfrak{sl}_m), A^{-n}(\mathfrak{sl}_{m}) \otimes L_r(\mathfrak{sl}_m) \right)\cong \\
&\qquad\qquad\qquad\cong \text{Com}\left(V^{-n+r}(\mathfrak{sl}_m),  \text{Com}\left(V^{-m}(\mathfrak{gl}_n) \otimes L_m(\mathfrak{gl}_r), \mathcal S(mn) \otimes \mathcal E(mr)\right)   \right)\\
&\qquad\qquad\qquad\cong  \text{Com}\left(V^{-n+r}(\mathfrak{sl}_m) \otimes V^{-m}(\mathfrak{gl}_n) \otimes L_m(\mathfrak{gl}_r), \mathcal S(mn) \otimes \mathcal E(mr)\right)  \\
&\qquad\qquad\qquad\cong  \text{Com}\left(V^{-n+r}(\mathfrak{gl}_m) \otimes V^{-m}(\mathfrak{sl}_n) \otimes L_m(\mathfrak{sl}_r) \otimes \mathcal H(1), \mathcal S(mn) \otimes \mathcal E(mr)\right) \\
&\qquad\qquad\qquad\cong  \text{Com}\left(  V^{-m}(\mathfrak{sl}_n) \otimes L_m(\mathfrak{sl}_r) \otimes \mathcal H(1), \text{Com}\left( V^{-n+r}(\mathfrak{gl}_m), \mathcal S(mn) \otimes \mathcal E(mr)\right)\right) \\
&\qquad\qquad\qquad\cong  \text{Com}\left(  V^{-m}(\mathfrak{sl}_n) \otimes L_m(\mathfrak{sl}_r) \otimes \mathcal H(1), A^m(\mathfrak{sl}_{r|n})\right)
\end{split}
\end{equation}
We thus have proven the level-rank duality theorem
\begin{thm} 
\label{thm:levelrank}
Let $r, n, m $ be positive integers. Then there exist vertex algebra extensions $A^{-n}(\mathfrak{sl}_{m})$ and $A^m(\mathfrak{sl}_{r|n})$ of homomorphic images $\widetilde V^{-n}(\mathfrak{sl}_{m})$ and $\widetilde V^m(\mathfrak{sl}_{r|n})$ of $V^{-n}(\mathfrak{sl}_{m})$ and $V^m(\mathfrak{sl}_{r|n})$
such that 
 the level-rank duality
\[
\text{Com}\left(V^{-n+r}(\mathfrak{sl}_m), A^{-n}(\mathfrak{sl}_{m}) \otimes L_r(\mathfrak{sl}_m) \right)\cong
\text{Com}\left(  V^{-m}(\mathfrak{sl}_n) \otimes L_m(\mathfrak{sl}_r) \otimes \mathcal H(1), A^m(\mathfrak{sl}_{r|n})\right)
\]
holds.
\end{thm}

\begin{rem}\label{rem:levelrank}
It is natural to ask if the statement of the Theorem can be improved, i.e. one could ask for a level-rank duality of the form
\[
\text{Com}\left(V^{-n+r}(\mathfrak{sl}_m), \widetilde  V^{-n}(\mathfrak{sl}_{m}) \otimes L_r(\mathfrak{sl}_m) \right)\overset{??}{\cong}
\text{Com}\left(  V^{-m}(\mathfrak{sl}_n) \otimes L_m(\mathfrak{sl}_r) \otimes \mathcal H(1), \widetilde  V^m(\mathfrak{sl}_{r|n})\right).
\]
\end{rem}

\begin{rem}
One might wonder if there are other levels $k$ for which the coset\newline  $\text{Com}\left(V^{k}(\mathfrak{sl}_n) \otimes V^{-k}(\mathfrak{sl}_r)\otimes  \mathcal H(1), V^{k}(\mathfrak{sl}_{n|r}) \right)$
coincides with a $\mathcal{W}$-algebra and indeed there are indications that these cosets sometimes coincide with rectangular $\mathcal{W}$-algebras of type $A$ \cite[Appendix D]{CH}.
\end{rem}

\appendix
\section{Decoupling relations}
\label{sec:app}
For the convenience of the reader we repeat the chosen conventions. 
\begin{align*}
	Q^+ 	&=G^{-,2} 
	\quad&
	Q^- 	&=G^{+,1} \quad&
	G^+ 	&=G^{+,2} 
	\quad&
	G^- 	&=G^{-,1} \\
	U_{n,0} &=\ :\partial^nJ^+J^-: 
	\quad&
	V_{n,0} &=\ :\partial^nG^{+}G^{-}: \quad&
	A_{n,0} &=\ :\partial^nJ^+G^{-}: 
	\quad&
	B_{n,0} &=\ :\partial^nJ^-G^{+}: 
\end{align*}
\begin{equation*}
\Sigma^{(0)\pm}_{2n,0}=\ :\partial^{2n}J^{\pm}J^{\pm}:, \qquad \Sigma^{(1)\pm}_{n,0}=:\partial^nJ^{\pm}G^{\pm}: , \qquad \Sigma^{(2)\pm}_{2n+1,0}=\ :\partial^{2n+1}G^{\pm}G^{\pm}:
\end{equation*}
All relations below were verified by computer \cite{doi:10.1142/S0129183191001001}. Note that applying the automorphism $\theta$ specified in the proof of Theorem \ref{THM: Main Theorem. minimal strong generators of cyclic orbifold Z/2Z} on the decoupling relation for $\Sigma^{(1)+}_{2,0}$ (see \ref{Appendix: Decoupling relation for Sigma^(1+)_(2,0)}) yields a decoupling relation for $\Sigma^{(1)-}_{2,0}$ that holds at all levels $k\neq4$.
\begin{equation}
\label{Appendix: Decoupling relation for U_{3,0}}
	\begin{split}
		0=&\frac{k}{3}U_{3,0}-\frac{k}{2}V_{2,0}+\partial U_{2,0} -\frac{1}{2}:U_{2,0}J: +\left(1+k\right)\partial V_{1,0} +:V_{1,0}J: \\
		&-:A_{1,0} Q^+: -:B_{1,0} Q^-: -\frac{1}{2}\partial^2U_{1,0} +:U_{1,0}T: -\frac{1}{2}:U_{1,0}\partial J: \\
		&+ \frac{1-k}{2}\partial^2 V_{0,0} -:\partial V_{0,0}J: -:V_{0,0}U_{0,0}: +:\partial A_{0,0}Q^+: 
		-:A_{0,0}B_{0,0}: \\
		&+:A_{0,0}\partial Q^+: +:\partial B_{0,0}Q^-: +:B_{0,0}\partial Q^-: 
		+\frac{1}{6}\partial^3U_{0,0} +:U_{0,0}\partial T: \\
		&+\frac{1+k}{6}\partial^3T +\frac{1}{2}:\partial^2TJ: -\frac{1}{2}:\partial^2Q^+ Q^-: 
		-:\partial Q^+\partial Q^-:\\ 
		&-\frac{1}{2}:Q^+\partial^2 Q^-: +\frac{3+k}{24}\partial^4 J +\frac{1}{12}:\partial^3JJ: .
	\end{split}
\end{equation}

\begin{equation}
\label{Appendix: Decoupling relation for Sigma^(1+)_(2,0)}
	\begin{split}
		0=&(4 - k) \Sigma^{(1)+}_{2,0} - 6 :J \Sigma^{(1)+}_{1,0}: - 2 :U_{0,0} \Sigma^{(1)+}_{0,0}: 
		- 2 :\partial J \Sigma^{(1)+}_{0,0}:\\ &+ 2 :\Sigma^{(0)+}_{0,0} B_{0,0}: + :\partial \Sigma^{(0)+}_{0,0} Q^+: 
	\end{split}
\end{equation}

\begin{equation}
\label{Appendix: Decoupling relation for U_(2,0)}
	\begin{split}
		0=&(16 - k) U_{2,0} - (8 + k) \partial U_{1,0} - 6 :J U_{1,0}: + \frac{2 + k}{2} \partial^2U_{0,0}  
		+ :J\partial U_{0,0}:\\&- :U_{0,0} U_{0,0}:- :\partial J U_{0,0}: - \frac{k}{6} \partial^3J 
		- \frac{1}{2} :\partial^2 J J: + :\Sigma^{(0)+}_{0,0}\Sigma^{(0)-}_{0,0}: = 0 
	\end{split}
\end{equation}

\begin{equation}
\label{Appendix: Decoupling relation for A_(2,0)}
	\begin{split}
		0=&\frac{16 - k}{2}A_{2,0} - 3:J A_{1,0}: - :U_{0,0} A_{0,0}: - :\partial J A_{0,0}: + 3 :U_{1,0} Q^-: \\
		&- :\partial U_{0,0} Q^-: + \frac{1}{2} :\partial^2J Q^-: + :\Sigma^{(0)+}_{0,0} \Sigma^{(1)-}_{0,0}:
	\end{split}
\end{equation}

\begin{equation}
\label{Appendix: Decoupling relation for B_(2,0)}
	\begin{split}
		0&=\frac{16 - k}{2} B_{2,0} + 2 (k-4) \partial B_{1,0} + 3 :J B_{1,0}: + (2 - k) \partial^2 B_{0,0} - 2 :J \partial B_{0,0}:\\
		& - :U_{0,0} B_{0,0}: + \frac{k-1}{3} \partial^3 Q^+ + :J \partial^2 Q^+: + 2 :U_{0,0} \partial Q^+: + 3 :U_{1,0} Q^+: \\
		&+ :\Sigma^{(1)+}_{0,0} \Sigma^{(0)-}_{0,0}: 
	\end{split}
\end{equation}

\begin{equation}
\label{Appendix: Decoupling relation for V^(2,0) Part 1}
	\begin{split}
		&k\left( (6 - k) V_{2,0} - :A_{0,0} \partial Q^+: + :\partial B_{0,0} Q^-:- \frac{\partial^2 Q^+ Q^-:}{2} : + :\Sigma^{(1)+}_{0,0} \Sigma^{(1)-}_{0,0}:  \right) \\
		& + (k-4) \left( \partial U_{2,0} + :\partial T U_{0,0}:+ \frac{(k+1)}{6}\partial^3T - :J \partial V_{0,0}:-:A_{0,0} B_{0,0}: \right.\\ &\left.+ \frac{1}{12} :\partial^3J J:+ \frac{1}{2} :\partial^2 T J:\right)- 4 :J V_{1,0}: + \frac{20-13 k}{2} \partial^2 U_{1,0}- 4 :T U_{1,0}: \\
		&+(2-k) \left( :J U_{2,0}:+:\partial J U_{1,0}: + 2:U_{0,0} V_{0,0}:\right) + (k^2- 5 k - 4) \partial V_{1,0} \\ 
		&+ \frac{4 + 4k - k^2}{2} \partial^2 V_{0,0}  + \frac{7k -10}{3} \partial^3 U_{0,0}   + (4 - 3 k) (:A_{1,0} Q^+: +  :B_{1,0} Q^-: )\\
		&+ \frac{k^2 - 26 k + 32}{24} \partial^4J =0\\
	\end{split}
\end{equation}

\begin{equation}
\label{Appendix: Decoupling relation for V^(2,0) Part 2}
	\begin{split}
		&(32 + 3 k) \left( \frac{k}{2} V_{2,0} - (1 + k)\left(  \partial V_{1,0} - \frac{\partial^2V_{0,0}}{2} + \frac{\partial^3T}{6}  \right) - :J V_{1,0}:\right.\\
		&\left. - :T U_{1,0}: + :J \partial V_{0,0}: + :U_{0,0} V_{0,0}: + :B_{1,0} Q^-: + :A_{0,0} B_{0,0}:\right. \\
		&\left. - \frac{1}{2} :\partial^2 T J: +  :A_{1,0} Q^+: - \frac{5}{6} \partial^3 U_{0,0} -  :\partial T U_{0,0}:\right)+ \frac{32 - 7 k}{2} :J U_{2,0}: \\
		&- (32 + 11 k + 2 k^2) \partial U_{2,0}  + \frac{160 + 19 k + 2 k^2}{2} \partial^2U_{1,0} + \frac{32 - k}{2} :\partial J U_{1,0}: \\
		&+ k\left( k :J \partial U_{1,0}: - 2 :U_{0,0} U_{1,0}: -  :\partial^2J U_{0,0}: +  :\partial\Sigma^{(0)+}_{0,0} \Sigma^{(0)-}_{0,0}: \right)  \\
		& - \frac{(8 + k) (5 k-32)}{24} \partial^4 J  \frac{32 + 7 k}{12} :\partial^3J J:  = 0
	\end{split}
\end{equation}

\section{Singular fields in $V^k(\mathfrak{n}_4)^{U(1)}$}
In this Appendix as well as the next one, the symbol $\propto$ is used to indicate equality up to multiplication by a non-zero constant over the base field which is assumed to be $\mathbb{C}$ throughout the article. 
\label{app: singular fields in U(1)-orbifold}

\subsection*{Level ${k=-\frac{5}{2}}$} 
	\begin{equation}
		\begin{split}
			U_{2,0} - V_{1,0} + \frac{1}{2} \partial V_{0,0} - :J V_{0,0}: - \partial U_{1,0} + \frac{3}{4} \partial^2 U_{0,0} - :U_{0,0} U_{0,0}: \\
			- :T U_{0,0}: + :\partial J U_{0,0}: - \frac{1}{2} :J J U_{0,0}: - \frac{5}{8} \partial^2T + \frac{3}{4} :T T: \\
			+ \frac{1}{2}:T \partial J:	- \frac{1}{4} :T J J:- \frac{1}{8} \partial^3J - \frac{1}{8} :\partial J \partial J: + \frac{1}{4} :\partial J J J: \\
			- \frac{1}{16}:J J J J:	+ \frac{3}{2}:\partial Q^+ Q^-:	- \frac{3}{2}:Q^+ \partial Q^-: + :J Q^+ Q^-: \\
			- 2 :Q^+ A_{0,0}: + 2 :Q^- B_{0,0}:
		\end{split}
	\end{equation}
\subsection*{Level ${k=-\frac{3}{2}}$} 
	\begin{equation}
		4 U_{0,0} - 2 T - 2 \partial J + :JJ:
	\end{equation}
	\begin{equation}
		-4 U_{1,0} + 6 \partial U_{0,0} + 4 :J U_{0,0}: - 3 \partial T - 4 :T J: + :J J J: + 2 :Q^+Q^-:
	\end{equation}
\subsection*{Level ${k=-\frac{4}{3}}$} 
	\begin{equation}
		\begin{split}
			\frac{1}{2}U_{2,0} + \frac{2}{3} V_{1,0} - \frac{1}{3} \partial V_{0,0} - \frac{1}{2} :J V_{0,0}:- \frac{1}{2} \partial U_{1,0} - \frac{1}{2} \partial^2U_{0,0} \\
			- 4 :U_{0,0} U_{0,0}:+ \frac{16}{3} :T U_{0,0}: + 4 :\partial J U_{0,0}: - 2 :J J U_{0,0}: + \frac{5}{12}\partial^2 T\\
			- \frac{23}{18} :T T:- \frac{8}{3} :T \partial J: + \frac{4}{3} :T J J: + \frac{37}{24} \partial^3 J- \frac{35}{12} :\partial^2JJ: \\
			- \frac{11}{8} :\partial J\partial J: + \partial J J J:	- \frac{1}{4}:J J J J: + \frac{4}{3}:\partial Q^+ Q^-: \\
			- \frac{4}{3} :Q^+\partial Q^-:	+ \frac{1}{2}:JQ^+Q^-: - :Q^+A_{0,0}: + :Q^- B_{0,0}:
		\end{split}
	\end{equation}
\subsection*{Level ${k=-\frac{2}{3}}$} 
	\begin{equation}
		\begin{split}
			-\frac{200}{27} V_{2,0} - \frac{32}{3} \partial U_{2,0} + \frac{4}{3} :J U_{2,0}: - \frac{88}{27} \partial V_{1,0} - \frac{88}{9} :J V_{1,0}: \\
			+ \frac{92}{9} \partial^2 U_{1,0}+ \frac{20}{3} :J \partial U_{1,0}:+ \frac{32}{3} :U_{0,0} U_{1,0}: - \frac{200}{9} :T U_{1,0}: \\
			+ \frac{8}{3} :\partial J U_{1,0}:+ \frac{8}{3} :J J U_{1,0}: - \frac{116}{9} \partial^2V_{0,0} + \frac{116}{9} :J \partial V_{0,0}: \\
			- \frac{8}{3} :U_{0,0} V_{0,0}: + \frac{112}{9} :T V_{0,0}: + \frac{28}{3} :\partial J V_{0,0}:	- \frac{14}{27}\partial^3 U_{0,0} \\
			- \frac{22}{3} :J\partial^2 U_{0,0}: - \frac{16}{3} :\partial U_{0,0} U_{0,0}:+ \frac{28}{9} :T\partial U_{0,0}: + \frac{2}{3} :\partial J \partial U_{0,0}: \\
			- \frac{4}{3} :J J \partial U_{0,0}:+ 16 :J U_{0,0} U_{0,0} - 16 :\partial T U_{0,0}: - \frac{112}{3} :T J U_{0,0}: \\
			+ \frac{56}{3} :U_{0,0} Q^+ Q^-:+ \frac{292}{9} :\partial^2J U_{0,0}: - 16 :\partial J J U_{0,0}: \\
			+ 8 :J J J U_{0,0}:	+ \frac{116}{27} \partial^3 T- \frac{38}{3} :\partial^2 T J: + \frac{56}{9} :\partial T \partial J: \\
			+ \frac{56}{3} :T T J:- \frac{112}{9} :T Q^+ Q^-: - \frac{742}{27} :T\partial^2 J: + \frac{56}{3} :T\partial J J: \\
			- \frac{28}{3} :T J J J:+ \frac{88}{9} :A_{1,0} Q^+: + \frac{88}{9} :B_{1,0} Q^-: - \frac{112}{9} :\partial A_{0,0} Q^+: \\
			+ 16 :A_{0,0} B_{0,0}: - \frac{112}{9} :A_{0,0} \partial Q^+: - \frac{28}{3} :J A_{0,0} Q^+: \\
			- \frac{112}{9} :\partial B_{0,0} Q^-: - \frac{112}{9} :B_{0,0} \partial Q^-: + \frac{28}{3} :J B_{0,0} Q^-: \\
			+ \frac{56}{9} :\partial^2 Q^+ Q^-: + \frac{112}{27} :\partial Q^+ \partial Q^-: - \frac{56}{9} :J\partial Q^+ Q^-: \\
			+ \frac{56}{9} :Q^+\partial^2 Q^-: + \frac{56}{9} :J Q^+\partial Q^-: - \frac{28}{3} :\partial J Q^+ Q^-: \\
			+ \frac{89}{18}\partial^4 J - \frac{17}{9} :\partial^3 JJ: - \frac{113}{9} :\partial^2J\partial J:+ \frac{190}{9} :\partial^2J J J: \\
			+ \frac{17}{3} :\partial J \partial J J: - 4 :\partial J J J J: + :J J J J J:
		\end{split}
	\end{equation}
\subsection*{Level ${k=-\frac{1}{2}}$} 
	\begin{equation}
		\begin{split}
			4 U_{1,0} - 5 V_{0,0} - 2 \partial U_{0,0} + 8 :J U_{0,0}: - 10 :T J:\\
			+ 7 \partial^2 J - 4 :\partial J J: + 2 :J J J: + 5 :Q^+Q^-:
		\end{split}
	\end{equation}
	\begin{equation}
		\begin{split}
			-3 U_{2,0} + 5 V_{1,0} + 4 \partial U_{1,0} + :J U_{1,0}: - \partial^2 U_{0,0} \\
			- :J \partial U_{0,0}:- 8 :U_{0,0} U_{0,0}:+ 10 :T U_{0,0}: + 5 :\partial J U_{0,0}: \\
			- 2 :J J U_{0,0}: + \frac{2}{3}\partial^3 J - 4 :\partial^2 J J: + 5 :Q^+ A_{0,0}:
		\end{split}
	\end{equation}
\subsection*{Level ${k=1}$}
		\begin{equation}
			(\Sigma^{(0)-}_{0,0})_{(2)} \Sigma^{(0)+}_{0,0}\propto-2 U_{0,0} + \partial J + :J J:
		\end{equation}
		\begin{align}
			\begin{split}
				(\Sigma^{(0)-}_{0,0,0})_{(3)} \Sigma^{(0)+}_{0,0,0}\propto\hspace{2pt} & 30 U_{1,0} - 18\partial U_{0,0} - 12:JU_{0,0}: \\
				& + :JJJ:+ 9:\partial J J: + 4\partial^2 J
			\end{split}
		\end{align}
		\begin{align}
			\begin{split}
				(\Sigma^{(0)-}_{0,0,0,0})_{(4)} \Sigma^{(0)+}_{0,0,0,0}\propto\hspace{2pt} & -480 U_{2,0} + 564\partial U_{1,0} - 174\partial^2 U_{0,0} \\
				& + 36:U_{0,0}U_{0,0}: + 288:JU_{1,0}: \\
				& - 180:J\partial U_{0,0}: - 36 :JJU_{0,0}: \\
				& - 72:\partial JU_{0,0}:+ :JJJJ:+ 30 :\partial JJJ: \\
				& +39 :\partial J\partial J: + 58 :\partial^2 JJ:+ 21 \partial^3 J
			\end{split}
		\end{align}
\subsection*{Level ${k=2}$} 
		\begin{align}
			\begin{split}
				(\Sigma^{(0)-}_{0,0,0})_{(3)} \Sigma^{(0)+}_{0,0,0}\propto\hspace{2pt} & 12 U_{1,0} - 6 \partial U_{0,0} - 6 :J U_{0,0}: \\
				& + \partial^2 J + 3 :\partial J J: + :J J J:
			\end{split}
		\end{align}
		\begin{align}
			\begin{split}
				(\Sigma^{(0)-}_{0,0,0,0})_{(4)} \Sigma^{(0)+}_{0,0,0,0}\propto\hspace{2pt} & -252 U_{2,0} + 276\partial U_{1,0} - 84\partial^2 U_{0,0} \\ 
				& + 18 :U_{0,0},U_{0,0}: + 156 :J U_{1,0}: \\
				& - 90 :J\partial U_{0,0}: - 24 :J J U_{0,0}: \\
				& - 30 :\partial J U_{0,0}:+ :JJJJ: \\
				& + 18:\partial JJJ: + 15 :\partial J \partial J: \\
				& + 25 :\partial^2 JJ: + 11\partial^3J
			\end{split}
		\end{align}
\subsection*{Level ${k=3}$} 
		\begin{align}
			\begin{split}
				(\Sigma^{(0)-}_{0,0,0,0})_{(4)} \Sigma^{(0)+}_{0,0,0,0}\propto &	-90 U_{2,0} + 90 \partial U_{1,0} + 60 :J U_{1,0}: - 27 \partial^2 U_{0,0} \\
				& - 30 :J \partial U_{0,0}: + 6 :U_{0,0} U_{0,0}: - 6 :\partial J U_{0,0}: \\
				&- 12 :J J U_{0,0}: + 4 \partial^3J + 7 :\partial^2J J: \\
				& + 3 \partial J\partial J: + 6 :\partial J J J: + :J J J J:
			\end{split}
		\end{align}
\subsection*{Level ${k=4}$} Note that the field $U_{3,0}$ appears in the expression below. Using \ref{Appendix: Decoupling relation for U_{3,0}} we see that the singular field induces a decoupling relation for the field $V_{2,0}$ in the simple quotient.
	\begin{align}
		\begin{split}
			(\Sigma^{(0)-}_{0,0,0,0,0})_{(4)} \Sigma^{(0)+}_{0,0,0,0,0} \propto\hspace{2pt} & 880 U_{3,0} - 1320\partial U_{2,0}	+ 1080\partial^2 U_{1,0} - 280\partial^3 U_{0,0} \\
			&- 120:U_{1,0}U_{0,0}: + 60:\partial U_{0,0}U_{0,0}:+ 30 :JU_{0,0}U_{0,0}:\\
			& - 660:JU_{2,0}:+ 660:J\partial U_{1,0}: - 240:J\partial^2 U_{0,0}:	\\
			& + 60:\partial JU_{1,0}: - 60:\partial J\partial U_{0,0}:+180:JJU_{1,0}: \\
			& - 90:JJ\partial U_{0,0}:- 20:JJJU_{0,0}:-30:\partial J J U_{0,0}: \\
			& + 10:\partial^2 J U_{0,0}: + :J J J J J:+ 10 :\partial J J J J: \\
			& + 15 :\partial J\partial J J:+ 25 :\partial^2 J J J: + 10 :\partial^2 J\partial J: \\
			&+ 55 :\partial^3J J:+51 \partial^4J
		\end{split}
	\end{align}

\section{Singular fields in $V^k(\mathfrak{n}_4)^{\mathbb{Z}/2\mathbb{Z}}$}
\label{app: singular fields in cyclic orbifold}
Note that some singular fields in this section involve the strong generator $U_{2,0}$ which decouples at all levels $k\neq16$ (see \ref{Appendix: Decoupling relation for U_(2,0)}).

\subsection*{Level ${k=-\frac{5}{2}}$} 
		\begin{equation}
			\begin{split}
				U_{2,0} - V_{1,0} + \frac{1}{2} \partial V_{0,0} - :J V_{0,0}: - \partial U_{1,0} + \frac{3}{4} \partial^2 U_{0,0} - :U_{0,0} U_{0,0}: \\
				- :T U_{0,0}: + :\partial J U_{0,0}: - \frac{1}{2} :J J U_{0,0}: - \frac{5}{8} \partial^2T + \frac{3}{4} :T T: \\
				+ \frac{1}{2}:T \partial J:	- \frac{1}{4} :T J J:- \frac{1}{8} \partial^3J - \frac{1}{8} :\partial J \partial J: + \frac{1}{4} :\partial J J J: \\
				- \frac{1}{16}:J J J J:	+ \frac{3}{2}:\partial Q^+ Q^-:	- \frac{3}{2}:Q^+ \partial Q^-: + :J Q^+ Q^-: \\
				- 2 :Q^+ A_{0,0}: + 2 :Q^- B_{0,0}:
			\end{split}
		\end{equation}
\subsection*{Level ${k=-\frac{3}{2}}$} 
		\begin{equation}
			4 U_{0,0} - 2 T - 2 \partial J + :JJ:
		\end{equation}
		\begin{equation}
			\begin{split}
			-4 U_{1,0} + 6 \partial U_{0,0} + 4 :J U_{0,0}: - 3 \partial T \\
			- 4 :T J: + :J J J: + 2 :Q^+Q^-:
			\end{split}
		\end{equation}
		\begin{equation}
			\begin{split}
			-5 \Sigma^{(1)+}_{1,0} + 3 \partial\Sigma^{(1)+}_{0,0} + :J\Sigma^{(1)+}_{0,0}: + 2 :Q^+ \Sigma^{(0)+}_{0,0}:
			\end{split}
		\end{equation}
\subsection*{Level ${k=-\frac{4}{3}}$}  
		\begin{equation}
			\begin{split}
				-1811 U_{2,0} + 12 V_{1,0} - 6\partial V_{0,0} - 9 :J V_{0,0}: + 691\partial U_{1,0} \\
				+ 630 :J U_{1,0}:- 18 :Q^+ A_{0,0}:+ 18 :Q^- B_{0,0}: - 44\partial^2 U_{0,0} \\
				- 105 :J\partial U_{0,0}:+ 33 :U_{0,0} U_{0,0}: + 96:T U_{0,0}:+ 177 :\partial J U_{0,0}: \\
				- 36 :J J U_{0,0}:+ \frac{15}{2}\partial^2 T - 23 :T T: - 48 :T\partial J: + 24 :T J J:\\
				- 105 :\Sigma^{(0)+}_{0,0}\Sigma^{(0)-}_{0,0}:+ \frac{53}{12}\partial^3 J- \frac{99}{4} :\partial J\partial J:+ 18 :\partial J J J: \\
				- \frac{9}{2} :J J J J:	+ 24 :\partial Q^+ Q^-: - 24 :Q^+\partial Q^-: + 9 :J Q^+ Q^-:
			\end{split}
		\end{equation}
\subsection*{Level ${k=-\frac{1}{2}}$} 
		\begin{equation}
			\begin{split}
				4 U_{1,0} - 5 V_{0,0} - 2 \partial U_{0,0} + 8 :J U_{0,0} - 10 :T J: \\
				+ 7 \partial^2 J - 4 :\partial J J: + 2 :J J J: + 5 :Q^+Q^-:
			\end{split}
		\end{equation}
		\begin{equation}
			\begin{split}
				-15 \Sigma^{(1)+}_{1,0} + 8 \partial \Sigma^{(1)+}_{0,0} + 2 :J \Sigma^{(1)+}_{0,0}: + 4 :Q^+ \Sigma^{(0)+}_{0,0}:
			\end{split}
		\end{equation}
\subsection*{Level ${k=1}$}
		\begin{equation}
			\Sigma^{(0)+}_{0,0}
		\end{equation}
		\begin{equation}
			(\Sigma^{(0)-}_{0,0})_{(2)} \Sigma^{(0)+}_{0,0}\propto-2 U_{0,0} + \partial J + :J J:
		\end{equation}
		\begin{align}
			\begin{split}
				(\Sigma^{(0)-}_{0,0,0})_{(3)} \Sigma^{(0)+}_{0,0,0}\propto\hspace{2pt} & 30 U_{1,0} - 18\partial U_{0,0} - 12:JU_{0,0}: \\
				& + :JJJ:+ 9:\partial J J: + 4\partial^2 J
			\end{split}
		\end{align}
\subsection*{Level ${k=2}$}
		\begin{align}
			\begin{split}
				(\Sigma^{(0)-}_{0,0,0})_{(3)} \Sigma^{(0)+}_{0,0,0}\propto\hspace{2pt} & 12 U_{1,0} - 6 \partial U_{0,0} - 6 :J U_{0,0}: \\
				& + \partial^2 J + 3 :\partial J J: + :J J J:
			\end{split}
		\end{align}
		\begin{equation}
			\Sigma^{(0)+}_{2,0} - :U_{0,0} \Sigma^{(0)+}_{0,0}: + 2 :\partial J \Sigma^{(0)+}_{0,0}:
		\end{equation}
\subsection*{Level ${k=3}$}
		\begin{equation}
			\begin{split}
				4 \Sigma^{(0)+}_{2,0} - 2 :U_{0,0} \Sigma^{(0)+}_{0,0}: + 5 \partial^2 \Sigma^{(0)+}_{0,0} \\
				- 8 :J \partial\Sigma^{(0)+}_{0,0}: - :\partial J \Sigma^{(0)+}_{0,0}: + 3 :J J \Sigma^{(0)+}_{0,0}:
			\end{split}
		\end{equation}



\footnotesize						

\begin{thebibliography}{ACKL17}

\bibitem[AKFPP16]{AKFPP} D. Adamovi\'{c}, V. G. Kac, P. Moseneder Frajria, P. Papi and O. Per\v{s}e, \emph{Finite vs infinite decompositions in conformal embeddings} Communications in Mathematical Physics, 348, 445-473 (2016).

\bibitem[ACKL17]{ACKL17}
T.~{Arakawa}, T.~{Creutzig}, K.~{Kawasetsu}, and A.~R. {Linshaw},
  \emph{{Orbifolds and cosets of minimal $\mathcal{W}$-algebras}}, Comm. Math.
  Phys. \textbf{355} (2017), 339--372.

\bibitem[ACL17]{2015arXiv151109143A}
T.~{Arakawa}, T.~{Creutzig}, and A.~R. {Linshaw}, \emph{{Cosets of
  Bershadsky-Polyakov algebras and rational $\mathcal{W}$-algebras of type
  $A$}}, Sel. Math. New Ser. \textbf{23} (2017), 2369--2395.

\bibitem[ACL19]{ACL19}
T.~{Arakawa}, T.~{Creutzig}, and A.~R. {Linshaw}, \emph{W-algebras as coset vertex algebras}, Invent. Math. 218, no. 1 (2019), 145-195.

\bibitem[{Ada}99]{1998math......9141A}
D.~{Adamovi\'{c}}, \emph{{Representations of the N = 2 superconformal vertex
  algebra}}, International Mathematics Research Notices \textbf{1999} (1999),
  no.~2, 61--79.

\bibitem[{Ada}16]{2014arXiv1407.1527A}
D.~{Adamovi\'{c}}, \emph{{A realization of certain modules for the $N=4$ superconformal
  algebra and the affine Lie algebra $A_2^{(1)}$}}, Transformation Groups
  \textbf{21} (2016), no.~2, 299--327.

\bibitem[{Al-}19]{2018arXiv180408189A}
M.~{Al-Ali}, \emph{The $\mathbb{Z}_2$-orbifold of the universal affine vertex
  algebra}, Journal of Pure and Applied Algebra \textbf{223} (2019), no.~12,
  5430 -- 5443.

\bibitem[AL17]{Al-Ali2017}
M.~{Al-Ali} and A.~R. {Linshaw}, \emph{The $\mathbb{Z}_2$-orbifold of the
  $\mathcal{W}_3$-algebra}, Communications in Mathematical Physics \textbf{353}
  (2017), no.~3, 1129--1150.

\bibitem[BL00]{Borisov2000}
L.~A. {Borisov} and A.~{Libgober}, \emph{{Elliptic genera of toric varieties
  and applications to mirror symmetry}}, Inventiones mathematicae \textbf{140}
  (2000), no.~2, 453--485.

\bibitem[BMR19]{Bonetti:2018fqz}
F.~{Bonetti}, C.~{Meneghelli}, and L.~{Rastelli}, \emph{{VOAs labelled by
  complex reflection groups and 4d SCFTs}}, JHEP \textbf{05} (2019), 155.

\bibitem[{Bor}86]{Borcherds:1983sq}
R.~E. {Borcherds}, \emph{{Vertex algebras, Kac-Moody algebras, and the
  monster}}, Proc. Nat. Acad. Sci. \textbf{83} (1986), 3068--3071.

\bibitem[{Bor}01]{Borisov}
L.~A. {Borisov}, \emph{{Vertex Algebras and Mirror Symmetry}}, Communications
  in Mathematical Physics \textbf{215} (2001), 517--557.

\bibitem[LZ95]{1995q.alg.....1014L}
B. H.~{Lian} and G.~J. {Zuckerman}, \emph{{Commutative quantum operator
  algebras}}, {Journal of Pure and Applied Algebra} \textbf{100} (1995), no.~1,
  117 -- 139.

\bibitem[CFL19]{CFL}
T.~{Creutzig}, B.~{Feigin}, and A.~R. {Linshaw}, \emph{{$N=4$ superconformal
  algebras and diagonal cosets}}, Int. Math. Res. Not. (2020), doi.org/10.1093/imrn/rnaa078.


\bibitem[CG17]{CGai}
T.~{Creutzig} and D.~{Gaiotto}, \emph{{Vertex Algebras for S-duality}}, Communications in Mathematical Physics, 379(3), 785-845.

\bibitem[CGL18]{CGL}
T.~{Creutzig}, D.~{Gaiotto}, and A.~R. {Linshaw}, \emph{{S-duality for the
  large $N=4$ superconformal algebra}}, Comm. Math. Phys. 374, no. 3 (2020), 1787-1808.

\bibitem[CH14]{Creutzig:2013mqa}
T.~{Creutzig} and G.~{H\"ohn}, \emph{{Mathieu Moonshine and the Geometry of K3
  Surfaces}}, Commun. Num. Theor. Phys. \textbf{08} (2014), 295--328.

\bibitem[CH19]{CH}
  T.~Creutzig and Y.~Hikida, \emph{Rectangular W algebras and superalgebras and their representations},
Phys. Rev. D \textbf{100} (2019) no.8, 086008.
 

\bibitem[CKL19]{CKL}
T.~{Creutzig}, S.~{Kanade}, and A.~R. {Linshaw}, \emph{{Simple current extensions beyond semi-simplicity}}, Commun. Contemp. Math. 22 (2020), no. 1, 1950001, 49 pp.


\bibitem[CKLR19]{CKLR}
T.~{Creutzig}, S.~{Kanade}, A.~R. {Linshaw}, and D.~{Ridout}, \emph{{Schur-Weyl
  duality for Heisenberg cosets}}, Transform. Groups \textbf{24} (2019),
  no.~2, 301--354.

\bibitem[CL18]{CLsub}
T.~{Creutzig} and A.~R. {Linshaw}, \emph{{Cosets of the $W^k(sl_4,
  f_{subreg})$-algebra }}, Contemp. Math. \textbf{711} (2018), 105--117.

\bibitem[CL19]{2014arXiv1407.8512C}
T.~{Creutzig} and A.~R. {Linshaw}, \emph{{Cosets of affine vertex algebras inside larger structures}},
  J. Algebra \textbf{517} (2019), 396--438.

\bibitem[CL20]{CL20}
T.~{Creutzig} and A.~R. {Linshaw}, \emph{{Trialities of $W$-algebras}}, arXiv:2005.10234.
  
  
\bibitem[{Cre}17]{Creutzig2017}
T.~{Creutzig}, \emph{{W-algebras for Argyres-Douglas theories}}, European
  Journal of Mathematics \textbf{3} (2017), no.~3, 659--690.

\bibitem[DLM96]{doi:10.1155/S1073792896000566}
C.~{Dong}, H.~{Li}, and G.~{Mason}, \emph{{Compact automorphism groups of
  vertex operator algebras}}, {International Mathematics Research Notices}
  \textbf{1996} (1996), no.~18, 913--921.

  
  
 \bibitem[DLWY10]{DLWY10} C. Dong, C. H. Lam, Q. Wang, and H. Yamada, \textit{The structure of parafermion vertex operator algebras}, J. Algebra 323 (2010), no. 2, 371-381.

\bibitem[EOT11]{Eguchi:2010ej}
T.~{Eguchi}, H.~{Ooguri}, and Y.~{Tachikawa}, \emph{{Notes on the K3 Surface
  and the Mathieu group $M_{24}$}}, Exper. Math. \textbf{20} (2011), 91--96.

\bibitem[ET88]{Eguchi:1988af}
T.~{Eguchi} and A.~{Taormina}, \emph{{On the Unitary Representations of $N=2$
  and $N=4$ Superconformal Algebras}}, Phys. Lett. \textbf{B210} (1988),
  125--132.

\bibitem[FB01]{frenkel2001vertex}
E.~{Frenkel} and D.~{Ben-Zvi}, \emph{Vertex algebras and algebraic curves},
  Mathematical surveys and monographs, American Mathematical Society, 2001.

\bibitem[FS07]{2003math......7181F}
E.~{Frenkel} and M.~{Szczesny}, \emph{{Chiral de Rham Complex and Orbifolds}},
  J. Algebraic Geom. \textbf{16} (2007), 599--624.

\bibitem[FST98]{1998JMP....39.3865F}
B.~L. {Feigin}, A.~M. {Semikhatov}, and I.~Yu. {Tipunin}, \emph{{Equivalence
  between chain categories of representations of affine sl(2) and N=2
  superconformal algebras}}, Journal of Mathematical Physics \textbf{39}
  (1998), 3865--3905.

\bibitem[{Hel}09]{2008arXiv0806.1021H}
R.~{Heluani}, \emph{{Supersymmetry of the Chiral de Rham Complex 2: Commuting
  Sectors}}, International Mathematics Research Notices \textbf{2009} (2009),
  no.~6, 953--987.

\bibitem[{Kac}01]{kac2001vertex}
V.~{Kac}, \emph{Vertex algebras for beginners}, University lecture series,
  American Mathematical Society, 2001.

\bibitem[{Kap}05]{Kapustin2005}
A.~{Kapustin}, \emph{{Chiral de Rham complex and the half-twisted
  sigma-model}}, arXiv:hep--th/0504074.

\bibitem[KR96]{Kac1996}
V.~{Kac} and A.~{Radul}, \emph{{Representation Theory of the Vertex Algebra
  $W_{1 + \infty}$}}, Transform. Groups \textbf{1} (1996), no.~1, 41--70.

\bibitem[{Li}04]{2002math......9310L}
H.~{Li}, \emph{{Vertex algebras and vertex Poisson algebras}}, Commun. Contemp. Math. \textbf{06} (2004), no.~01, 61--110.


\bibitem[{Lin}17]{L17}
A.~R. {Linshaw}, \emph{{Universal two-parameter $\mathcal{W}_{\infty}$-algebra
  and vertex algebras of type $\mathcal{W}(2,3,\dots, N)$}}, arXiv:1710.02275, to appear in Compos. Math.

\bibitem[LL07]{2006math......5174L}
B.~H. {Lian} and A.~R. {Linshaw}, \emph{{Howe pairs in the theory of vertex
  Algebras}}, J. Algebra \textbf{317} (2007), 111--152.

\bibitem[LSS15]{LSS} A.~R. {Linshaw}, G. Schwarz and B. Song, \emph{Arc spaces and the vertex algebra commutant problem}, Adv. Math. 277 (2015), 338-364.

\bibitem[{Mal}99]{Maldacena:1997re}
J.~M. {Maldacena}, \emph{{The Large $N$ limit of superconformal field theories
  and supergravity}}, Int. J. Theor. Phys. \textbf{38} (1999), 1113--1133.

\bibitem[MS99a]{1999math......1065M}
F.~{Malikov} and V.~{Schechtman}, \emph{{Chiral de Rham complex. II}}, Am.
  Math. Soc. Transl. \textbf{194} (1999), 149--188.

\bibitem[MS99b]{1999math......5008M}
F.~{Malikov} and V.~{Schechtman}, \emph{{Chiral Poincar\'e duality}}, Math. Res. Lett. \textbf{6}
  (1999), 533.

\bibitem[MS03]{Malikov2003}
F.~{Malikov} and V.~{Schechtman}, \emph{{Deformations of Vertex Algebras, Quantum Cohomology of Toric
  Varieties, and Elliptic Genus}}, Comm. Math. Phys.
  \textbf{234} (2003), no.~1, 77--100.

\bibitem[MSV99]{1999CMaPh.204..439M}
F.~{Malikov}, V.~{Schechtman}, and A.~{Vaintrob}, \emph{{Chiral de Rham
  Complex}}, Comm. Math. Phys. \textbf{204} (1999),
  439--473.

\bibitem[OS14]{OstrikSun} V. Ostrik and M. Sun, \emph{Level-Rank Duality via Tensor Categories}, Comm. Math. Phys. 326 (2014), no. 1, 49-61. 


\bibitem[{Son}15]{2013arXiv1312.7386S}
B.~{Song}, \emph{{The Global Sections of the Chiral de Rham Complex on a Kummer
  Surface}}, International Mathematics Research Notices \textbf{2016} (2015),
  no.~14, 4271--4296.

\bibitem[{Son}16]{2016arXiv160903688S}
B.~{Song}, \emph{{Vector bundles induced from jet schemes}}, arXiv:1609.03688.

\bibitem[{Son}18]{2018arXiv180910320S}
B.~{Song}, \emph{{The global sections of chiral de Rham complexes on compact
  Ricci-flat K\"ahler manifolds}}, arXiv:1809.10320.

\bibitem[{Thi}91]{doi:10.1142/S0129183191001001}
K.~{Thielemans}, \emph{{A Mathematica\texttrademark package for computing
  operator product expansions}}, International Journal of Modern Physics C
  \textbf{02} (1991), 787.

\end{thebibliography}
\bibliographystyle{amsalpha}
\nocite{*}							

\providecommand{\bysame}{\leavevmode\hbox to3em{\hrulefill}\thinspace}
\providecommand{\MR}{\relax\ifhmode\unskip\space\fi MR }
\providecommand{\MRhref}[2]{%
  \href{http://www.ams.org/mathscinet-getitem?mr=#1}{#2}
}
\providecommand{\href}[2]{#2}

\end{document}